\documentclass[english,11pt]{amsart}
\usepackage{comment}
\usepackage[latin9]{inputenc}
\usepackage{float}
\usepackage{mathtools}
\usepackage{amstext}
\usepackage{amsthm}
\usepackage{amsmath}
\usepackage{amssymb}
\usepackage{algorithm}
\usepackage{geometry}
\usepackage[backref=page,colorlinks,citecolor=blue,bookmarks=true]{hyperref}\hypersetup{linkcolor=blue,  citecolor=blue, urlcolor=blue}

\usepackage[nobysame,abbrev,alphabetic]{amsrefs}

\usepackage{microtype}
\usepackage{tikz}
\usetikzlibrary{cd}
\usetikzlibrary{shapes}

\usepackage{nicefrac,xfrac}

\floatname{algorithm}{Algorithm}

\makeatletter
   
\numberwithin{equation}{section}
\numberwithin{figure}{section}
\theoremstyle{plain}
\newtheorem{thm}{\protect\theoremname}[section]
\DeclareMathOperator*{\Ex}{\mathbb{E}}
\DeclareMathOperator*{\Pro}{\mathbb{P}}

\newtheorem{cor}[thm]{Corollary}%{\protect\Corollaryname}
\newtheorem{obs}[thm]{Observation}
\newtheorem*{thm*}{\protect\theoremname}
\theoremstyle{definition}
\newtheorem{problem}[thm]{\protect\problemname}
\newtheorem*{problem*}{Problem}
\theoremstyle{remark}
\newtheorem*{rem*}{\protect\remarkname}
\theoremstyle{remark}
\newtheorem{rem}[thm]{\protect\remarkname}

\theoremstyle{definition}
\newtheorem{defn}[thm]{\protect\definitionname}
\theoremstyle{plain}
\newtheorem{prop}[thm]{\protect\propositionname}
\theoremstyle{plain}
\newtheorem{fact}[thm]{\protect\factname}
\theoremstyle{definition}

\theoremstyle{plain}
\newtheorem{lem}[thm]{\protect\lemmaname}
\theoremstyle{plain}
\newtheorem{claim}[thm]{Claim}%{\protect\claimname}
\theoremstyle{plain}
\usepackage[backref=page,colorlinks,citecolor=blue,bookmarks=true]{hyperref}\hypersetup{linkcolor=blue,  citecolor=blue, urlcolor=blue}
\usepackage[fontsize=11pt]{scrextend}

\usepackage[nobysame,abbrev,alphabetic]{amsrefs}
\usepackage{algorithmic}

\usepackage{microtype}

\floatname{algorithm}{Algorithm}

\let\originalleft\left
\let\originalright\right
\renewcommand{\left}{\mathopen{}\mathclose\bgroup\originalleft}
\renewcommand{\right}{\aftergroup\egroup\originalright}

\makeatother

\usepackage{babel}
\makeatletter
\addto\extrasfrench{%
   \providecommand{\fg}{\ifdim\lastskip>\z@\unskip\fi~\frqq}%
}

\makeatother
\addto\captionsenglish{\renewcommand{\definitionname}{Definition}}
\addto\captionsenglish{\renewcommand{\factname}{Fact}}
\addto\captionsenglish{\renewcommand{\lemmaname}{Lemma}}
\addto\captionsenglish{\renewcommand{\problemname}{Problem}}
\addto\captionsenglish{\renewcommand{\propositionname}{Proposition}}
\addto\captionsenglish{\renewcommand{\remarkname}{Remark}}
\addto\captionsenglish{\renewcommand{\theoremname}{Theorem}}
\addto\captionsfrench{}
\addto\captionsfrench{\renewcommand{\definitionname}{Definition}}
\addto\captionsfrench{\renewcommand{\factname}{Fait}}
\addto\captionsfrench{\renewcommand{\lemmaname}{Lemme}}
\addto\captionsfrench{\renewcommand{\problemname}{Problème}}
\addto\captionsfrench{\renewcommand{\propositionname}{Proposition}}
\addto\captionsfrench{\renewcommand{\remarkname}{Remarque}}
\addto\captionsfrench{\renewcommand{\theoremname}{Théorème}}
\providecommand{\definitionname}{Definition}
\providecommand{\factname}{Fact}
\providecommand{\lemmaname}{Lemma}
\providecommand{\problemname}{Problem}
\providecommand{\propositionname}{Proposition}
\providecommand{\remarkname}{Remark}
\providecommand{\theoremname}{Theorem}
\providecommand{\examplename}{Example}

\newcommand{\eps}{\varepsilon}

\newcommand{\Id}{{\rm Id}}
\newcommand{\enc}{\texttt{enc}}

\newcommand{\cF}{\mathcal{F}}
\newcommand{\cE}{\mathcal{E}}
\newcommand{\FF}{\mathbb{F}}
\newcommand{\cA}{\mathcal{A}}

\newcommand{\cT}{\mathcal{T}}

\newcommand{\cC}{\mathcal{C}}

\newcommand{\Img}{\textrm{Im}}
\newcommand{\Ker}{\textrm{Ker}}

\title[Non-commutative error correcting codes]{Non-commutative error correcting codes\\ and proper subgroup testing}

\author[M.\ Chapman]{Michael Chapman}
\address{Michael Chapman\hfill\break
	Courant Institute of Mathematical Sciences\hfill\break
	New York University,\hfill\break 251 Mercer St, New York, NY 10012, USA.}
\email{mc9578@nyu.edu}

\author[I.\ Dinur]{Irit Dinur}
\address{Irit Dinur\hfill\break
	Weizmann institute of Science\hfill\break
	Rehovot, Israel.}
\email{irit.dinur@weizmann.ac.il}

\author[A.\ Lubotzky]{Alexander Lubotzky}
\address{Alexander Lubotzky\hfill\break
	Weizmann institute of Science\hfill\break
	Rehovot, Israel.}
\email{alex.lubotzky@mail.huji.ac.il}

\begin{document}

\begin{abstract}
  Property testing has been a  major area of research in computer science in the last three decades. By property testing we refer to an ensemble of problems, results and algorithms which enable to deduce global information about some data by only reading small random parts of it. In recent years, this theory found its way into group theory (see \cites{becker2022testability,CL_part1,BCLV_subgroup_tests} and the references therein), mainly via \emph{group stability}. 

  In this paper, we study the following problem: Devise a randomized algorithm that given a subgroup $H$ of $G$, decides whether $H$ is the whole group or a proper subgroup, by checking whether a single (random) element of $G$ is in  $H$. The search for such an  algorithm boils down to the following purely group theoretic problem: For $G$ of rank $k$, find a small as possible \emph{test subset} $A\subseteq G$ such that for every proper subgroup $H$, $|H\cap A|\leq (1-\delta)|A|$ for some absolute constant $\delta>0$, which we call the \emph{detection probability} of $A$.  It turns out that the search for sets $A$ of size linear in $k$ and constant detection probability is a non-commutative analogue of the classical search for families of good error correcting codes. This paper is devoted to proving that such test subsets \textbf{exist}, which implies good \emph{universal  error correcting codes}   exist --- providing a far reaching generalization of the classical result of Shannon \cite{shannon1948mathematical}.
In addition, we study this problem in certain subclasses of groups --- such as abelian, nilpotent, and finite solvable groups --- providing different constructions of test subsets for these subclasses with various qualities. Finally, this generalized theory of  non-commutative error correcting codes suggests a plethora of interesting  problems and research directions.
  
\end{abstract}

\maketitle
\tableofcontents
\section{Introduction}\label{sec:intro}

Let $G$ be a group of rank $k$, and $A$ a finite (multi-)subset\footnote{By a multi-subset, we mean a subset which is a multi-set, namely a finite unordered collection of elements of $G$ which may include repetitions.} of $G$. The \emph{detection probability} of $A$ is 
  \begin{equation}\label{eq:defn_detection_prob}
       \delta(A;G)=\inf\left\{1-\frac{|A\cap H|}{|A|}\ \middle|\ H\lneq G\right\}.\footnote{Note that the same quantity is calculated by running only over $H$ which are maximal subgroups of $G$.}
  \end{equation}
The theme of this paper is finding \textbf{small} subsets with \textbf{large} detection probability for various classes of groups. 
The main result of this paper is:
\begin{thm}[Main Theorem]\label{conj:main}
    There exist absolute constants $\delta_0>0$ and $C_0> 1$, such that every group $G$ of rank $k$ has a (multi-)subset $A$ with 
    \[
|A|\leq C_0\cdot k\quad \textrm{and}\quad \delta(A;G)\geq \delta_0.
    \]
\end{thm}
The search for small subsets with large detection probability, and in particular  Theorem \ref{conj:main}, is interesting from three different perspectives:
Robust generation of groups;
Property testing in group theory;
 Non-commutative error correcting codes.

  \subsection*{I. Robust generation}
  Note  that if the subgroup generated by $A$ is proper, then the detection probability $\delta(A;G)=0$. Hence, for $\delta(A;G)$ to be positive, $A$ needs to be a generating set. Actually, this is a special kind of generating set --- \textbf{every} $1-\delta(A;G)$ portion of the elements of $A$ generate $G$. Hence, we seek generating sets not much larger than a minimal one, such that every large enough subset of them is still generating.
\subsection*{II. Proper Subgroup Testing}
A property tester is a randomized algorithm that aims to distinguish between objects that satisfy a certain property, and objects that are \emph{far}  from elements satisfying the property, by querying the object at a few random positions.\footnote{For a thorough introduction to property testing, see \cite{Goldreich}. For a somewhat general property testing framework, specifically aimed at mathematicians, see Section 2 of \cite{CL_part1}.}
In our context, the tester $\cT$ gets as input a group $G$\footnote{Encoded, for example, as a finitely presented group with generators and relations --- i.e., a pair consisting of a positive integer  $k$ (in binary) which encodes the number of generators, and a finite list of words in the free group $\cF_k$ with basis $B=\{x_1,...,x_k\}$ that represents the relations. } 
of rank $k$ and access to a black box function  $h\colon G\to \{0,1\}$. The function $h$ is guaranteed to be the indicator of some subgroup $H\leq G$. The goal of the tester is to distinguish between the case when $H$ is the whole group $G$, and when $H$ is a proper subgroup of it.\footnote{This problem may seem a bit out of context. It arises naturally in the computational model of \emph{Subgroup Tests} suggested in \cite{BCLV_subgroup_tests}, specifically in the design of PCPs within the model.}

Sets $A$ as in Theorem \ref{conj:main} suggest the following $1$-query\footnote{This is the term for how many evaluation points of the black box function $h$ does the tester uses when it runs. The general case of $q$-queries --- for constant $q$ or some $q$ growing with $k$ --- is also interesting.} tester for this problem:\footnote{Since the tester  $\cT$ is associated with $A$, we often refer to such sets $A$ as \emph{test subsets}.} $\cT$ chooses a word $w\in A$ uniformly at random and checks the value of $h$ at $w$; if $h(w)=0$, namely $w\notin H$, it concludes that ``\emph{$H$ is a proper subgroup}''; otherwise $h(w)=1$, namely $w\in H$, in which case $\cT$ did not find a sign that $H$ is proper and concludes ``\emph{$H$ is the whole group}''.  
It is straightforward to see that 
\[
\Pro\left[\ \cT\ \textrm{concludes\ }H=G\ \right]=\frac{|A\cap H|}{|A|}.
\]
Thus, when $H=G$, the tester never errs. On the other hand, when $H\lneq G$, the tester \emph{detects} it with probability of at least $\delta(A;G)$, which is why this quantity is called the \emph{detection probability}.\footnote{The parameter $1-\delta(A;G)$ is often called the \emph{soundness} parameter, when viewing this as a decision problem.} The  size of $A$  is comparable to the amount of random bits  $\cT$  uses,   which is considered a scarce resource, further motivating making $A$ small without hindering the detection probability.

As this is a computational problem, we care about various other aspects of the tester and the test subset: What is the running time of $\cT$?  Is the test subset $A$ constructed explicitly or using probabilistic methods? How long are the elements of $A$ as words in the generators of $G$?  etc.
This motivates the following notion which will be used throughout the paper:
\begin{defn}
    Let $G$ be a group and $S$ a generating set of size $k$. The length $L(A)=L_{G,S}(A)$ of a finite subset $A\subseteq G$ with respect to $S$ is the smallest $\ell$ such that every element in $A$ can be written as a word of length at most $\ell$ in $S\cup S^{-1}$.
\end{defn}
\subsection*{III. Non-commutative error correcting codes}
Let us observe what our Main Theorem says about the group $\FF_p^k$, the $k$-dimensional vector space over the field of order $p$, where $p$ is a prime integer. Let $A=\{\alpha_1,...,\alpha_n\}$ be the subset of $\FF_p^k$ guaranteed by Theorem \ref{conj:main}, namely $n\leq C_0\cdot k$ and $\delta(A;\FF_p^k)\geq \delta_0$.
Define a linear map $\enc\colon \FF_p^k\to \FF_p^n$ by letting 
\[
 v\in \FF_p^k \colon \ \ 
\enc(v)=(\langle v,\alpha_i\rangle)_{i=1}^n
\]
  where $\langle v,\alpha\rangle=\sum_{i=1}^k v_i\alpha_i $ is a fixed bilinear form, and let $\cC=\Img(\enc)\leq \FF_p^n$ be the image subspace. As $\delta(A;\FF_p^k)\geq \delta_0>0$, $A$ is generating $\FF_p^k$. Hence, $\enc$ is injective and $\dim(\cC)=k$. Moreover, as
$
v^\perp=\{\alpha\in \FF_p^k\mid \langle v,\alpha\rangle=0\}
$ is a co-dimension $1$ subspace of $\FF_p^k$ 
  for every $\vec 0\neq v\in \FF_p^k$, and $
    A\cap v^\perp=\{\alpha_i \in A\mid \langle v,\alpha_i\rangle=0\}
$, we can deduce that 
\[
\begin{split}
    \delta_0\leq 1-\frac{|A\cap v^\perp|}{|A|}=\frac{|\{\alpha_i\in A\mid \langle v,\alpha_i\rangle \neq0\}|}{n}=\frac{w_H(\enc(v))}{n},
\end{split}
\]
where $w_H(u)=|\{i\in [n]\mid u_i\neq 0\}|$ is the \emph{Hamming weight} of the vector $u\in \FF_p^n$. All in all, $\cC$ is a $k$-dimensional subspace of $\FF_p^n$ such that every non-zero vector $u\in \cC$ satisfies $w_H(u)\geq \delta_0\cdot n$. 

Subspaces of $\FF_p^n$ of dimension $k$ with minimal non-zero Hamming weight $d$ are called \emph{linear} $[n,k,d]_p$-codes. The quantity $\nicefrac{k}{n}$ is referred to as the \emph{rate} of the code, while $\nicefrac{d}{n}$ is its \emph{normalized distance}. In his seminal paper \cite{shannon1948mathematical}, Shannon defined a \emph{good code} to be an infinite family of $[n,k,d]_p$-codes (where $n\to\infty$ and) with constant rate and normalized distance, and showed that a \emph{random} code is good,\footnote{More on that in the beginning of Section \ref{sec:randomized_construction}.} laying the foundations to the theory of error correcting codes. As the analysis shows, the codes $\cC$ constructed using  Theorem \ref{conj:main} are good codes with rate of at least $\nicefrac{1}{C_0}$ and normalized distance of at least $\delta_0$.

Our Main Theorem can thus be seen as a vast generalization of Shannon's Theorem, from the family $\frak{G}=\{\FF_p^k\}_{k\in \mathbb{N}}$ to the family $\frak{U}$ of \textbf{all groups}, which includes (in particular) all infinite groups --- this point is of special interest, as there is no natural random choice of a finite test subset $A$ for an infinite group $G$. Borrowing the notations from classical codes, we define:
\begin{defn}[$\frak{G}$-codes]\label{def:G_codes}
    Let $\frak{G}$ be a class of groups. An $[n,k,d]_\frak{G}$-code is an assignment that associates with every $G\in \frak{G}$ of rank $k$ a multi-subset $A$ of $G$ of size $n$ such that $\delta(A;G)\geq \nicefrac{d}{n}$. If $\frak{U}$ is the class of all groups, we call an $[n,k,d]_{\frak{U}}$-code a \emph{universal} code. We often use the term $\frak{G}$-code for an $[n,k,d]_{\frak{G}}$-code without specifying the parameters $n,k,d$. Such a code is called \emph{good} if the quantities $\nicefrac{k}{n}$ and $\nicefrac{d}{n}$ are bounded from below by some positive constants independent of $k$.
\end{defn}

\begin{cor}
In this language, Theorem \ref{conj:main} shows the existence of \textbf{good universal codes}.
\end{cor}

This perspective on the search for small test subsets with large detection probability suggests studying the optimal tradeoff between these quantities. Recall that there are various upper  and lower bounds (cf. \cite{Venkat_bounds_ECC})  on the best possible \emph{rate} $\rho=\nicefrac{k}{n}$ given a fixed \emph{normalized distance} $\delta=\nicefrac{d}{n}$ for error correcting codes over $\FF_p$. E.g., using an elementary sphere packing argument, every $[n,k,d]_p$-code satisfies $\nicefrac{k}{n}+H_p(\nicefrac{d}{2n})\leq 1+o(1)$, where $H_p$ is the $p$-ary entropy function. Gilbert \cite{gilbert1952comparison} and Varshamov \cite{varshamov1957estimate} provided a lower bound on this quantity by showing that a randomly chosen linear code achieves $\nicefrac{k}{n}+H_p(\nicefrac{d}{n})\geq 1-o(1)$.  This bound is called the \emph{GV-bound}. The constants $C_0$ and $\delta_0$ we provide in Theorem \ref{conj:main} are far from achieving the GV-bound. This leads to two interesting problems: Are there stricter upper bounds for universal codes that are absent in the finite field setup? Are there \emph{universal codes} that achieve the GV-bound?

\begin{rem}
    Due to the connection to classical error correcting codes described above, our constructions of universal codes (and $\frak{G}$-codes for other classes $\frak{G}$) are all inspired by various constructions and techniques from the classical theory of error correction.
\end{rem}

\subsection*{Our results}

A priori, it is not  clear that there exists a constant $\delta_0>0$ such that every finitely generated group $G$ has a subset $A$ with $\delta(A;G)\geq \delta_0$, without any assumptions on the size of $A$ except it being finite. Resolving this is our first result.
\begin{thm}[Universal Hadamard code]\label{thm:intro_exp_sized_set}
    Let $G$ be a group of rank $k$. Then, there exists an explicit set $A\subseteq G$ such that $|A|=2^k$,$L(A)=k$ and $\delta(A;G)\geq \nicefrac{1}{2}$. 
\end{thm}
Our construction in Theorem \ref{thm:intro_exp_sized_set} can be seen as a non-commutative analogue of the Hadamard code. It gives words of length linear in $k$, which is optimal (up to the specific constant) as Proposition \ref{prop:running_time} shows.
By applying careful random sampling  on this construction, combined with code composition techniques, we deduce our second result.
\begin{thm}\label{thm:intro_poly_sized_set}
    Every group of rank $k$ has a subset $A$ with $|A|=O(k\log^{13} k)$, $L(A)=O(k\log k\log\log k)$ and $\delta(A;G)\geq \nicefrac{1}{32}$.
\end{thm} 
In fact, this is one of an infinite sequence of results that depend on the number of code composition iterations one applies --- for the above, we compose twice. In each such iteration, the (asymptotic) size of $A$ gets better, while  both the detection probability and the constants involved in the bounds on $A$ get worse. For the exact result, see Section  \ref{sec:iterative}. 

By using the iterative encoding technique of \cite{spielman1995linear}, we provide in Section \ref{sec:Spielman} good universal codes, which  proves our Main Theorem  \ref{conj:main}:
\begin{thm}[Main Theorem: Existence of good universal  codes]\label{thm:Spielman_construction}
    There are explicit good universal  codes with polynomial word length. Namely, every group $G$ of rank $k$ has an explicit subset $A$  with $|A|\leq 8k$, $L(A)={\rm poly}(k)$ and $\delta(A;G)=\Omega(1).$
\end{thm}

As in \cite{spielman1995linear}, the above construction uses almost optimal \emph{unique neighbor expanders} (which can be derived from \emph{lossless expanders}) as infrastructure.  Such (explicit) families were constructed in \cite{alon2002explicit,capalbo2002randomness}, and later also in \cites{asherov2023bipartite,cohen2023hdx, hsieh2023explicit,  golowich2024new}. Regarding the involved constants, even if one is willing to drop explicitness, the  probabilistic method argument (which we include in  Appendix \ref{sec:existence_lossless}) gives rise to $\delta(A;G)\geq 2^{-130}$,   which is not a  constant to be proud of, and the maximal length of the words in $A$ is $O(k^{10})$ in this case. In principle, one can extract the appropriate detection constant from the above explicit constructions of lossless expanders, as well as an explicit bound on the degree of the polynomial bounding the length of the words, but they will be worse than the $2^{-130}$ and $10$ cited above. We hope that the future will bring better methods, which will lead to reasonable detection constants. It also leaves open the following problem: Are there test sets $A$ with $L(A)=O(k),\ |A|=O(k)$ and $\delta(A;G)=\Omega(1)$? We show in Proposition \ref{prop:running_time} that $L(A)=\Omega(\delta(A;G)k)$, so this would be optimal. 
\begin{rem}\label{rem:one_trick}
    The  observation driving Theorems \ref{thm:intro_exp_sized_set}, \ref{thm:intro_poly_sized_set} and \ref{thm:Spielman_construction}  is quite straightforward: If $w$ is a word with letters drawn from $B=\{x_1,...,x_t\}\subseteq G$, and all letters in $w=\prod_{j=1}^\ell x_{i_j}^{\eps_j}$ ($i_j\in [t],\ \eps_{j}\in\{\pm 1\}$) are in $H\leq G$ except for exactly \textbf{one}  --- namely, there is an index $1\leq r\leq \ell$ such that $x_{i_j}\in H$ for every $j\neq r$ and $x_{i_r}\notin H$ --- then $w\notin H$.
\end{rem}

\begin{rem}\label{rem:intro_uniform_generation_of_universal_codes}
    In Theorems \ref{thm:intro_exp_sized_set}, \ref{thm:intro_poly_sized_set} and \ref{thm:Spielman_construction} we actually prove something slightly stronger. Note that the detection probability and size\footnote{Here it is important that we treat $A$ as a multi-set and not just a set, so its size is preserved by quotients.} of $A$ are preserved by quotients (Fact \ref{fact:basic_properties}). We construct an appropriate set $A$ for the free group, and as all rank $k$ groups are quotients of it, this construction is \emph{uniform} in the following sense: Regardless of the specific $G$, as long as it is of rank $k$ and a set of $k$ generators is provided, our construction is some fixed collection of words in these generators. \emph{One set of words to rule all rank $k$ groups}.
\end{rem}

Though  Theorem \ref{thm:Spielman_construction} already provides a good \textbf{universal}  code, we provide additional constructions of good codes for some sub-classes of groups using different methods. In Section \ref{sec:abelian}, a construction is provided for the class of abelian groups.
\begin{thm} [Good abelian  codes]\label{thm:abelian_case}
    Every abelian group $G$ of rank $k$ has an explicit subset $A$ with $|A|=O(k)$, with constant detection probability $\delta(A;G)=\Omega(1)$ and with at most exponential length  $L(A)\leq 2^{O(k)}$. 
\end{thm}

This construction  uses Tanner codes with underlying unique neighbor expanders, and thus the specific constants depend again on the choice of these expanders. For concreteness, allowing a probabilistic argument using (for example) Lemma 1.9 from \cite{Hoory_Linial_Wigderson}, one can attain $|A|=4k$ and $\delta(A;G)\geq \nicefrac{1}{320}$.  Any result on abelian groups holds automatically also for nilpotent groups, as Corollary \ref{cor:nilpotent} demonstrates. 
\begin{rem}\label{rem:simultanious2}
    Similar to the content of Remark \ref{rem:intro_uniform_generation_of_universal_codes}, Theorem \ref{thm:abelian_case} is resolved by a construction over free abelian groups. For every $k$, our construction provides an encoding matrix with integer coefficients,   that induces good codes over \textbf{all} primes $p$ simultaneously. We do not know of previous works studying this property.
\end{rem}

In Section \ref{sec:PMSG}, we prove the following:
\begin{thm}[Good finite solvable  codes]\label{thm:intro_solvable_groups}
    Every finite solvable group $G$ of rank $k$ has a subset $A$ with $|A|\leq 85 k$ and with detection probability of at least $\nicefrac{1}{10}$.
\end{thm}
In fact, we  prove something much more general regarding classes of finite groups with restrictions on their sections.  This proof is again probabilistic, and we cannot even bound the lengths of the elements in $A$ as words in the given generators. On the other hand, as the theorem suggests, we get effective bounds on the size of $A$ and its detection probability.

Let us summarize the results appearing in this paper in the following table:
\begin{center}
\begin{tabular}{ |c|c|c|c|c|c|c| } 
 \hline
 Class of groups & $|A|=n(k)$ & $\delta(A;G)=\delta$ & $L(A)$ & Explicit? & GV bound? & Where proved\\ 
 \hline\hline
 Finite abelian groups & $O(k)$ & $\Omega(1)$ & $O(k)$ & No & Yes & Section \ref{sec:fin_abelian_gps}\\ 
 All groups & $2^k$ & $\nicefrac{1}{2}$ & $k$ & Yes & No & Theorem \ref{thm:exp_set}\\ 
  All groups & $O(k\log^{13}k)$ & $\nicefrac{1}{32}$ & $O(k\log k\log\log k)$ & No & No & Corollary \ref{cor:two_iterations}\\ 
   All groups & $4k$ & $\Omega(1)$ & $\textrm{poly}(k)$ & Yes&  No & Corollary 
   \ref{cor:blabla}\\ 
    Abelian groups & $O(k)$ & $\Omega(1)$ & $\textrm{exp}(k)$ & Yes & No & Proposition  
   \ref{prop:matrices_to_abelian_case}\\
    Finite solvable groups & $85k$ & $\nicefrac{1}{10}$ & Unbounded & No& No & Theorem  
   \ref{thm:fin_sol_gps}\\
 \hline
\end{tabular}
\end{center}
\ \\
We did not try to optimize our parameters. Some improvements are probably possible. What is really interesting is to understand the tradeoff  between $|A|$, $\delta(A;G)$ and $L(A)$ --- specifically, as we mentioned before, are there  universal codes that achieve the Gilbert--Varshamov bound, or is there an obstruction in the non-commutative case (or even the general abelian case) that does not exist in the classical theory? 
More on this (and other problems) in the open problem section, Section \ref{sec:open_problems}. 

\subsection*{Structure of the paper}
  Section \ref{sec:proper_subgps_and_ECCs} provides basic observations and analysis of test subsets. Specifically, it relates it to the theory of error correcting codes.  Section \ref{sec:non-explicit_solution} is devoted to the construction of universal codes, and includes the proofs of Theorems \ref{thm:intro_exp_sized_set}, \ref{thm:intro_poly_sized_set} and \ref{thm:Spielman_construction}.
  Section \ref{sec:abelian} is devoted to 
(infinite) abelian codes  and contains the proof of Theorem \ref{thm:abelian_case}. In Section \ref{sec:PMSG} we study codes over certain classes of profinite groups with restricted  maximal subgroup growth, which in turn proves Theorem \ref{thm:intro_solvable_groups}.  A discussion on running time appears in Section \ref{sec:running_time}, while open problems and further research directions appear in Section \ref{sec:open_problems}.

\subsection*{Acknowledgements}
We would like to thank Lewis Bowen, Oded Goldreich, Noam Kolodner, Sandro Mattarei, Carlo Pagano,  Doron Puder,   Thomas Vidick and Avi Wigderson for helpful discussions along the preparation of this work. We also want to thank the \emph{Midrasha on Groups} in Weizmann institute for its role in forming this collaboration.  

Michael Chapman acknowledges with gratitude the Simons Society of Fellows and is supported by a grant from the Simons Foundation (N. 965535). 
Irit Dinur is supported  by ERC grant 772839, and ISF grant 2073/21.
Alex Lubotzky is supported by the European Research Council (ERC)
under the European Union's Horizon 2020 (N. 882751), and by a research grant from the Center for New Scientists at the Weizmann Institute of Science.

\section{Proper subgroup testing and Error correcting codes}\label{sec:proper_subgps_and_ECCs}

\subsection{Analysis} 
Let $G$ be a finitely generated group. The \emph{rank}  of $G$, $k=\textrm{rank}(G)$, is the  minimal  size of a generating set of it.\footnote{It is common to denote the rank of $G$ by $d(G)$. But, because of the connections to error correcting codes, we prefer to keep $d$ for the Hamming metric and distance of codes, and choose $k$ for the rank.} 
  Let $A$ be a finite multi-set of elements of a group $G$.
As was defined in the introduction \eqref{eq:defn_detection_prob}, 
   the \emph{detection probability} of $A$ is 
    \[
        \delta(A;G)=\inf\left\{1-\frac{|A\cap H|}{|A|}\middle|H\lneq G\right\}.
    \]
    Note that $1-\frac{|A\cap H|}{|A|}=\frac{|A-H|}{|A|}$, where $A-H=\{a\in A\mid a\notin H\}$ is the set difference.

\begin{fact}\label{fact:basic_properties}
Let $G$ be a group and $A$ a finite multi-subset of $G$.
    \begin{enumerate}
        \item The detection probability $\delta(A;G)$ is positive if and only if $A$ generates $G$.
        \item If $\pi\colon G\to G'$ is an epimporphism, then $\pi(A)$ has detection probability which is at least as good as $A$, namely
        $\delta(\pi(A);G')\ge \delta(A;G).$
        \item An epimporphism $\pi\colon G\to G'$  is called a  Frattini extension if its kernel is contained in the Frattini subgroup of $G$. Equivalently, every maximal subgroup of $G$ contains $\ker \pi$. For Frattini extensions, we have $\delta(\pi(A);G')=\delta(A;G)$. 
    \end{enumerate}
\end{fact}
\begin{proof}[Proof sketch]
\ 
    \begin{enumerate}
        \item Note that $\langle A\rangle$ is a subgroup of $G$.
        \item By the correspondence theorem (fourth isomorphism theorem), for every $H'\lneq G'$ there is a $\ker \pi \leq H\lneq G$ such that $\pi (H)=H'$. Furthermore, as we use multi-sets, $|A|=|\pi(A)|$ and $|A\cap H|=|\pi(A\cap H)|\leq|\pi(A)\cap H'|$. Since this goes over all subgroups of $G'$ (but not necessarily all subgroups of $G$), we get the conclusion. 
        \item Note again by the correspondence theorem, that if $M\lneq G$ is a maximal subgroup, and $\ker \pi \subseteq M$, then $\pi(M)\lneq G'$. Also, as  $\ker \pi \subseteq M$, $\pi(A\cap M)=\pi(A)\cap \pi(M)$. This provides the reverse inequality.
    \end{enumerate}
\end{proof}

\begin{rem}\label{rem:density}\
\begin{enumerate}
\item 
 Fact \ref{fact:basic_properties} shows why constructing linearly sized test sets with constant detection probability for free groups would imply the existence of such sets for all groups.
    \item The free group $\cF_k$ is a quotient of $\cF_{s}$ for every $s\geq k$. So, by Fact \ref{fact:basic_properties},  it is enough to construct linearly sized test sets with constant detection probability for  \emph{dense enough} collections of $k\in \mathbb{N}$, and not necessarily to all of them.
\end{enumerate}
\end{rem}

\subsection{Error correction}\label{sec:ECC}
Let $p$ be a prime number, and let $\FF_p=\{0,1,...,p-1\}$ be the field with $p$ elements. Given two vectors $u,v\in \FF_p^n$, the Hamming distance between them is $d_H(u,v)=|\{i \in [n]\mid u_i\neq v_i\}|$. The Hamming weight $w_H(v)$ of a vector $v\in \FF_p^n$ is its distance to the all $0$'s vector $\vec 0$.
A (linear) $[n,k,d]_p$-code $\mathcal{C}$  is a $k$-dimensional
linear subspace of $\mathbb{F}_{p}^{n}$, such that the Hamming distance between any two vectors in the code is
 at least $d$. The parameter $n$ is  the \emph{length} of the code. The parameter $k$ is the \emph{dimension} of the code, while $\nicefrac{k}{n}$ is  the \emph{rate} of the code.  Lastly, the parameter $d$ is the \emph{distance} of the code, while $\nicefrac{d}{n}$ is its \emph{normalized distance}.   A family of codes is called \emph{good} if they have a uniform lower bound on their normalized distance and rate.
 
Let $\enc \colon\mathbb{F}_{p}^{k}\to\mathbb{F}_{p}^{n}$
be a linear embedding such that ${\rm Im}(\enc)=\mathcal{C}.$
Let $\pi_{i}\colon\mathbb{F}_{p}^{n}\to\mathbb{F}_{p}$ be the projection
on the $i^{{\rm th}}$ coordinate, and define ${\enc}_{i}\colon\mathbb{F}_{p}^{k}\to\mathbb{F}_{p}$
to be the composition $\pi_{i}\circ{\enc}.$ Define a bilinear
product 
\begin{align*}
\forall v,w & \in\mathbb{F}_{p}^{k}\ \colon\ \left\langle v,w\right\rangle =\sum_{i=1}^{k}v_{i}w_{i}.
\end{align*}
Given $\alpha\in\mathbb{F}_{p}^{k},$ define the functional $\varphi_{\alpha}\colon\mathbb{F}_{p}^{k}\to\mathbb{F}_{p}$
by $\varphi_{\alpha}(v)=\left\langle \alpha,v\right\rangle ,$ and
let $\alpha^{\perp}={\rm ker}\varphi_{\alpha}.$ Also, for every functional
$\varphi\colon\mathbb{F}_{p}^{k}\to\mathbb{F}_{p}$, there exists an
$\alpha\in\mathbb{F}_{p}^{k}$ such that $\varphi=\varphi_{\alpha}.$
Since $\left\{ {\enc}_{i}\right\} _{i\in[n]}$ are functionals,
there exists a collection $\left\{ v_{i}\right\} _{i\in n}\subseteq\mathbb{F}_{p}^{k}$
such that ${\enc}_{i}=\varphi_{v_{i}}$.

Assume  $A\subseteq\cF_k$ is the set guaranteed by Theorem  \ref{conj:main}, where $\cF_k$ is the free group on basis $B=\{x_1,...,x_k\}$, and recall  that $\delta_0$ is the lower bound on its detection probability. By applying on $\cF_k$ abelianization mod $p$, $\Phi_p\colon \cF_k\to \FF_p^k$, the (multi-)set $Z=\Phi(A)$ satisfies both $|Z|=|A|\leq C_0\cdot k$ and $\delta(Z;\FF_p^k)\geq \delta_0$  (by Fact \ref{fact:basic_properties}).
As we remarked when defining the detection probability \eqref{eq:defn_detection_prob}, we may assume the proper linear subspace $H$ is maximal, i.e.,
it is of co-dimension $1$. Let $\alpha\in\mathbb{F}_{p}^{k}$ be a vector such that $H=\alpha^{\perp}$.
I.e., if $h$ is the indicator function  of $H$, then $h(v)=1$ if and only if $\langle\alpha,v\rangle=0$.
Therefore, the fact the detection probability of $Z$ is bounded from below by $\delta_0$ is equivalent to the condition 
\[
\forall \vec{0}\neq\alpha\in\mathbb{F}_{p}^{k}\ \colon \ \ \left|\left\{ v\in Z\mid\left\langle \alpha,v\right\rangle \neq 0\right\} \right|\geq\delta_0|Z|.
\]
Let $|Z|=n$,  $\enc\colon \FF_p^k\to \FF_p^n$ be $\enc(\alpha)=(\langle\alpha,v\rangle)_{v\in Z}$, and  $\cC=\Img(\enc)$. Then $\cC$ is an $[n,k,\delta_0 n]_p$-code. Thus:
\begin{cor}\label{cor:prop_subgp_test_implies_codes}
    A collection of sets as in  Theorem \ref{conj:main} induces, in a uniform way,  a collection of good codes on $\FF_p^k$ for every $p$.
\end{cor}
The observation of Corollary \ref{cor:prop_subgp_test_implies_codes} is  the key to most of our results ---  we use  constructions and methods from the theory of error correcting codes as inspiration to try to resolve the opposite direction, i.e., to build \textbf{test subsets} out of codes, or phrased differently --- \textbf{universal codes} using classical coding theory techniques. 
\\

If the above perspective on our problem as non-commutative  error correction was not motivating enough, we provide the test case of \textbf{finite} abelian groups, where the connection between the problems is quite tight (Section \ref{sec:abelian} resolves the general infinite abelian groups case).

\subsection{Finite abelian groups}\label{sec:fin_abelian_gps}
As we have shown above, if we want to find a subset $A$ of  $\FF_p^k$ of size $n$ with detection probability $\delta>0$, then it is equivalent to finding an error correcting code $[n,k,\delta n]_p$. In particular, we can use the Gilbert--Varshamov bound, and find sets of size  $n\approx\frac{k}{1-H_p(\delta)}$ with detection probability $\delta>0$, where $H_p$ is the $p$-ary entropy function $H_p(x)=x\log_p(p-1)-x\log_px-(1-x)\log_p(1-x)$. 

Let $G$ be a finite abelian group. Every such group is isomorphic to a (finite) product of cyclic groups of prime power order $\nicefrac{\mathbb{Z}}{p^k\mathbb{Z}}$. The Frattini subgroup in this case is $\nicefrac{p\mathbb{Z}}{p^k\mathbb{Z}}$. Hence, by item $(3)$ of Fact \ref{fact:basic_properties}, we can assume $G$ is a (finite) product of finite vector spaces $\FF_p^{r(G,p)}$.

Now, let $G$ and $G'$ be two finite abelian groups such that $\gcd(|G|,|G'|)=1$, and let $A=\{a_1,...,a_n\}\subseteq G, B=\{b_1,...,b_n\}\subseteq G'$ be ordered (multi-)subsets the same size with detection probability $\delta$. Let $C\subseteq G\times G'$ be the diagonal (multi-)subset associated with $A$ and $B$,
\[
\forall i\in[n]\ \colon \ \ c_i=(a_i,b_i)\in C.
\]
Then, $C$ has  detection probability $\delta$. This is because for every maximal subgroup $H\leq G\times G'$, $\nicefrac{G\times G'}{H}$ is a group of prime order, which is either co-prime to $|G|$ or to $|G'|$. Therefore, either  $ \{\Id\}\times G'$ or $G\times \{\Id\}$ is contained in $H$. This in turn implies that either $\pi_G(H)$ is a maximal subgroup of $G$, or that  $\pi_{G'}(H)$ is a maximal subgroup of $G'$,  where $\pi_G$ (respectively $\pi_{G'}$) is the projection to the left (respectively right)  coordinate. All in all, either  $|C\cap H|=|A\cap \pi_G(H)|\leq (1-\delta)|A|=(1-\delta)|C|$, or  $|C\cap H|=|B\cap \pi_{G'}(H)|\leq (1-\delta)|B|=(1-\delta)|C|$.

Combining the above observations, we deduce a Gilbert--Varshamov type bound for finite abeilan groups:
\begin{cor}
    Let $G$ be a finite abelian group. Modulo its Frattini subgroup, $G$ is a product of vector spaces $\FF_p^{r(G,p)}$. Then, for every $0<\delta<1-\frac{1}{\min \{p\mid r(G,p)\neq 0\}}$, we can  find a subset of $G$ of size  (approximately) $\displaystyle{\max_{p\ \textrm{prime}}}\frac{r(G,p)}{1-H_{p}(\delta)}$ and detection probability $\delta$.
Since the rank of $G$  is $k=\max(r(G,p))$, and  for every $0<\delta<\nicefrac{1}{2}$ we have $H_2(\delta)\geq H_p(\delta)$, $G$ has a subset of size (approximately) $\frac{k}{1-H_2(\delta)}$ with detection probability $\delta$.
\end{cor}

\section{Universal codes}\label{sec:non-explicit_solution}

As mentioned in Remark \ref{rem:one_trick} in  the Introduction, the following observation is repeatedly used in this section:
\begin{obs}\label{obs:one_trick}
    Let $B=\{x_1,...,x_t\}$ be a subset of $G$, and let $w=\prod_{j=1}^\ell x_{i_j}^{\eps_j}$ be a word in $B$, namely $i_j\in [t]$ and $\eps_{j}\in \{\pm 1\}$ for every $1\leq j\leq \ell$. Let $H$ be a subgroup of $G$. Assume there is an $1\le r\leq \ell$ such that $x_{i_r}\notin H$ while $x_{i_j}\in H$ for every $j\neq r$. Then $w\notin H$.
\end{obs}

\begin{rem}\label{rem:intersection_of_tuple_and_set}
    We often want to think of finite multi-sets $A$ of $G$ both as sets (which allows us to take intersections of them with other subsets) and with a certain order (which makes them into tuples). This creates situations where we use notations as $\vec g\cap H$ for a tuple of elements from $G$ and a subgroup $H$. By this we mean that we forget about the order of $\vec g$, which leaves us with a finite multi-set  $A$ of elements in $G$, and then the intersection is the multi-set of all elements  from $A$ that are also in $H$.
\end{rem}

\subsection{An exponential relaxation.}
 At first glance, it may not  be clear that there are finite sets of $\cF_k$ with detection probability bounded  by  $\delta>0$ which is independent of $k$. It turns out that sets of size exponential in $k$ suffice.
 
\begin{thm}\label{thm:exp_set}
There is a set $A\subseteq \cF_k$ of size $|A|=2^{k},$ such that
$\delta(A;\cF_k)\ge\frac{1}{2}$.
\end{thm}
\begin{rem}
    The following construction is a non-commutative version of the Hadamard code. Namely, if you run the sets constructed in the following proof through the process described in Section \ref{sec:proper_subgps_and_ECCs}, you get  (classical) codes; choosing  $p=2$ will result  with the Hadamard code.
\end{rem}

\begin{proof}[Proof of Theorem \ref{thm:exp_set}]
Let $B=\{x_{1},...,x_{k}\}$ be a basis for $\mathcal{F}_{k}$.
For every subset $S\subseteq B$, let $s$ be its indicator
function, and let 
\begin{equation}\label{eq:notation_of_x_S}
    x_{S}=x_{1}^{s(x_{1})}\cdot x_{2}^{s(x_{2})}\cdot...\cdot x_{k}^{s(x_{k})}.
\end{equation}
Namely, for every subset $S$ of $B$ we take $x_{S}$ to be the product
of the elements in $S$ according to the order defined by their index.
Now, let $A=\left\{ x_{S}\right\} _{S\subseteq B}.$ Clearly $|A|=2^{k}.$
Let $H$ be a proper subgroup of $\mathcal{F}_{k}$. Let $x_{i}$
be the generator with the smallest index not belonging to $H$, namely
$x_{i}\notin H$ but $x_{j}\in H$ for every $j<i$. For every subset $S$
of $B$ not containing $x_{i},$  let $S'=S\cup\{x_{i}\}$.
Then, at most one of $x_{S}$ and $x_{S'}$ can be in $H$. This is
because $x_{S'}x_{S}^{-1}$ is a conjugate of $x_{i}\notin H$ by
an element of $H$ --- the product $x_{1}^{s(x_{1})}\cdot...\cdot x_{i-1}^{s(x_{i-1})}$
is in $H$ regardless of what $S$ is --- and in particular cannot be
in $H$ (this is a simple version of Observation \ref{obs:one_trick}). Since $S\leftrightarrow S'$ is a perfect matching of the subsets of $B$, we deduce that $\frac{|A\cap H|}{2^k}\leq \frac{1}{2}$, and thus $\delta(A;\cF_k)\geq \frac{1}{2}$.
\end{proof}

\begin{rem} Avi Wigderson  has mentioned to us that  the technique used in Theorem \ref{thm:exp_set} --- namely, ordering the generators and taking subwords in this order --- was utilized in various works, and specifically in the following two influential papers by Babai--Szemer{\'e}di \cite{babai1984complexity} and Alon--Roichman \cite{alon1994random}.
\end{rem}

\subsection{A randomized construction}\label{sec:randomized_construction}
In the theory of classical error correcting codes, one of the first observations is that random codes are good\footnote{This was mentioned both in the Introduction as Shannon's theorem, as well as  in Section \ref{sec:ECC} as the Gilbert--Varshamov \textbf{lower} bound.}. We spell out the main proof ideas of this fact for the case of $\FF_2$:  If we sample a vector in $\FF_2^k$ uniformly at random, it has a probability of $\frac{1}{2}$ to be out of a given co-dimension $1$ subspace $H\subseteq \FF_2^k$. Hence, by standard concentration of measure techniques (see Lemma \ref{lem:Chernoff_bound}), if we sample $n$ vectors from $\FF_2^k$, the probability that less than $\nicefrac{1}{4}$ of them are not in $H$ decreases exponentially in $n$. Since there are only  $2^k-1$ different co-dimension $1$ subspaces, one can apply a union bound to deduce that a collection of $n=O(k)$ random vectors will provide an encoding map $\enc\colon \FF_2^k\to \FF_2^n$, as in Section \ref{sec:ECC}, whose image is a  code with normalized distance of at least $\nicefrac{1}{4}$ almost surely. 

One can hope to mimic the above argument by sampling $O(k)$ elements of the set $\widetilde B=\{x_S\}_{S\subseteq B}$ from  Theorem \ref{thm:exp_set} to produce the required universal codes a la Definition \ref{def:G_codes}. The issue with this approach is in the last step of the analysis: While $\FF_2^k$ has $2^k-1$ maximal subgroups, $\cF_k$ has infinitely many. We overcome this problem by clustering subgroups of $\cF_k$ into finitely many bins, actually, $2^k-1$ many bins. This clustering of subgroups is according to their \textbf{syndrome} in the basis of the free group. The resulting clusters are amenable to a  probabilistic argument, though not with respect to the uniform distribution over $\widetilde B$, as is explained soon.

Let us spell out the above plan with more details, beginning with our clustering mechanism. For every subgroup $H$ and set $B=\{x_1,...,x_k\}\subset G$, the \emph{syndrome} of $H$ in $B$ is
\begin{equation}\label{eq:defn_syndrome}
Synd(H;B)=B-H=\{x_i\in B\mid x_i\notin H\}.    
\end{equation}
The syndrome with respect to the basis $B$ of $\cF_k$ distributes its infinitely many subgroups  into finitely many bins: For every $C\subseteq B$, let 
\begin{equation}\label{eq:defn_Sigma_C}
    \Sigma_C=\bigcup_{Synd(H';B)=C}{H'}.
\end{equation}
Now, for every $H\lneq \cF_k$, we have $Synd(H;B)\neq \emptyset$, and thus every proper subgroup is in some $\Sigma_C$ for $\emptyset \neq C\subseteq B$. There are only  $2^k-1$ such non-trivial syndromes, and thus we are back to familiar grounds --- if we are able to find an $A$ such that for every $\emptyset\neq  C\subseteq B$, a constant fraction of $A$ avoids $\Sigma_C$, then we are done. 

To that end, we need the following observation: As in \eqref{eq:notation_of_x_S}, for every subset $S\subseteq B$ with characteristic function $s$,  let $x_S=\prod x_i^{s(x_i)}\in \widetilde B$. Note that 
\begin{equation}\label{eq:syndrome_implies_notin}
    |S\cap Synd(H;B)|=1 \implies x_S\notin H,
\end{equation} 
which is another special case of  Observation \ref{obs:one_trick}.
Therefore, one seeks a sampling procedure on subsets $S$ of $B$ with a high chance that the resulting subset intersects a specific syndrome $C$ exactly  \textbf{once}, and thus $x_S\notin \Sigma_C$. This leads us to the final key observation:
    Let $H$ be a proper subgroup with syndrome $C$, and let $\sigma=|C|>0$. Let $0<p<1$. If we sample $S\subseteq B$ such that each $x_i$ is included in $S$ with probability $p$ independently of the others, then
    \begin{equation} \label{eq:syndrome_intersects_once}
        \Pro[|S\cap C|=1]= \sigma p(1-p)^{\sigma-1}.
    \end{equation}
    In particular, if $p= \frac{1}{\sigma}$, then $\sigma p(1-p)^{\sigma-1}\geq e^{-1}$. This calculation is actually quite robust, namely, even for $\frac{1}{2\sigma}\leq p\leq \frac{2}{\sigma}$ we get a constant lower bound on $\Pro[|S\cap C|=1]$. Therefore, we sample in the following way: 
    For every $\ell\leq \log k$, we sample $O(k)$ many subsets $S\subseteq B$ of size $2^\ell$ uniformly and independently. That way, for every $\emptyset\neq C\subseteq B$ there exists an $\ell$ such that a constant fraction of the sampled sets $S$ of size $2^\ell$ that were  sampled satisfy $|S\cap C|=1$. Thus, the detection probability of this sampled set is $\Omega(\nicefrac{1}{\log k})$. This is the content of Theorem \ref{thm:first_iteration}.
But the goal is to have a constant detection probability. To resolve that, in Proposition \ref{prop:amplification_to_const}, we provide an amplification trick that replaces  the guaranteed $A$ with a not much larger $A'$, with the benefit that $\delta(A';\cF_k)$ is $\Omega(1)$. This amplification trick is analogous to composing the random code defined in Theorem \ref{thm:first_iteration} with the Hadamard code of Theorem \ref{thm:exp_set}. In Section \ref{sec:iterative}, we show that more iterations of code  compositions lead to somewhat better parameters, from which  we deduce Theorem \ref{thm:intro_poly_sized_set}.
\begin{rem}
    Our sampling scheme may resonate with a reader  which is familiar with the proof of Bourgain's embedding theorem \cite{bourgain1985lipschitz}.
\end{rem}
    
    To be able to leverage the above analysis into a construction, we need to  recall the Chernoff \cite{chernoff1952measure} (Hoeffding \cite{Hoeffding}) bound.
    
\begin{lem}[Chernoff bound, cf. Theorem 3.3 in \cite{Chernoff_ODonnell} or Theorems 4.4 and 4.5 in \cite{mitzenmacher2017probability}]\label{lem:Chernoff_bound} Let $n$ be a positive integer.
    Let $\{X_i\}_{i=1}^n$ be a collection of  independent and identically distributed random variables, where $0\leq X_i\leq 1$ and $\Ex[X_i]=\mu$.  Then, for every $0\leq \eps\leq 1$, we have 
    \begin{align}
        &\Pro\left[\sum X_i> (1+\eps) n\mu\right]\leq e^{-\eps^2n\mu/3},\\
        &\Pro\left[\left|\sum X_i-n\mu\right|> \eps n\mu\right]\leq 2e^{-\eps^2n\mu/3}.
    \end{align}
\end{lem}
\begin{thm}\label{thm:first_iteration}
        There is a  set $A$ in $\cF_k$ of size 
         $|A|=432k\log k$, such that for every subset $\emptyset \neq C\subseteq B$,  
            \begin{equation}\label{eq:syndrome_detection_prob}
                \frac{|A\cap \Sigma_C|}{|A|}\leq 1-\frac{1}{12\log k}.
            \end{equation}
    \end{thm}
\begin{rem}\ 
    \begin{enumerate}
        \item Since every $H\lneq \cF_k$ is contained in $\Sigma_{Synd(H;B)}$, \eqref{eq:syndrome_detection_prob}  implies $\delta(A;\cF_k)\geq \frac{1}{12\log k}$.
        \item The parameter $\delta=\frac{1}{12\log k}$ is actually not that bad for applications. By allowing the tester to query the black box indicator function $h$ at ${\rm poly}(\log k)$ many positions instead of $1$, which keeps the query complexity \emph{efficient}, one can leverage this to a constant detection probability.  But, as we prove in Proposition \ref{prop:amplification_to_const} and Corollary \ref{cor:resolving_main_1}, there are constructions that resolve the problem without increasing the query complexity.
        \item As $A$ is a subset of $\widetilde B$, its length is at most $k$.
    \end{enumerate}
\end{rem}
\begin{proof}[Proof of Theorem \ref{thm:first_iteration}]
The proof idea is as follows: We choose for every $\ell\leq \log k$ a collection of $O(k)$ random words from $\widetilde B=\{x_S\}_{S\subseteq B}$ of length $2^\ell$. Since each such word has a constant chance of being out of $\Sigma_C$ for every $2^\ell \leq |C|\leq  2^{\ell+1}$, we apply Chernoff and a union bounds to deduce what is needed.
\\

    For every $1\leq \ell \leq \log k$, sample subsets $\{S(\ell,j)\}_{j=1}^{432k}$  of $B$ independently as follows: Each $x_i\in B$ belongs to $S(\ell,j)$ with probability $\frac{1}{2^\ell}$, independently of the other elements of $B$. Recall the notation  from \eqref{eq:notation_of_x_S} and let $A_\ell=\{x_{S(\ell,j)}\mid  1\leq j\leq 432k\}$ and $A=\bigcup A_\ell$. Let 
    $\emptyset\neq C \subseteq B, \sigma=|C|$ and $\ell=\lceil\log(\sigma)\rceil$.
    Then, by using   \eqref{eq:defn_Sigma_C},\eqref{eq:syndrome_implies_notin} and \eqref{eq:syndrome_intersects_once}, for every $1\leq j\leq 432k$, we have
    \[
    \begin{split}
        \Pro[x_{S(\ell,j)}\notin \Sigma_C]
        &\geq \Pro[|S(\ell,j)\cap C|=1]\\
        &=\underbrace{\frac{\sigma}{2^{\ell}}}_{\geq \nicefrac{1}{2}}\left(\underbrace{1-\frac{1}{2^{\ell}}}_{\geq 1-\frac{1}{\sigma}}\right)^{\sigma-1}\\
        &\geq \nicefrac{1}{2e}.
    \end{split}
    \]
    Let $X_j$ be the characteristic random variable of the event $x_{S(\ell ,j)}\in \Sigma_C$. Then, $\mu =\Ex[X_j]\leq 1-\frac{1}{2e}<\frac{5}{6}$, and also $|A_{\ell}\cap \Sigma_C|=\sum X_j$. Now, if we choose $\eps>0$ such that  $(1+\eps)\mu=\nicefrac{11}{12}$, then $$\eps> \eps\mu> \nicefrac{11}{12}-\nicefrac{10}{12}=\nicefrac{1}{12}.$$ 
    Therefore, by applying Lemma \ref{lem:Chernoff_bound}, we can deduce that 
    \[
        \Pro\left[|A_\ell\cap\Sigma_C|\geq \frac{11}{12}\cdot 432k\right]\leq e^{-\frac{432}{3\cdot 12^2}k}=e^{-k},
    \]
    and
    \[
        \Pro\left[\exists \emptyset \neq C\subseteq B\colon |A_\ell\cap\Sigma_C|\geq \frac{11}{12}\cdot 432k\right]<2^k\cdot e^{-k}<1.
    \]
    Hence, there is a choice of $A$ such that for every $\emptyset\neq C\subseteq B$ we have
    \[
\begin{split}
    \frac{|A\cap\Sigma_C|}{|A|}\leq 1-\frac{|A_\ell -\Sigma_C|}{432k\log k}\leq 1-\frac{1}{12\log k},
\end{split}
    \]
    and the proof is finished.
\end{proof}

The next proposition allows us to amplify the detection probability of certain constructions to a constant without enlarging the test subset $A$ too much (It will be applied on the outcome of Theorem \ref{thm:first_iteration} with parameter $\delta =\nicefrac{1}{12\log k}$).
\begin{prop}[Amplification to constant by sub-sampling]\label{prop:amplification_to_const}
    Let $A$ be a set in $\cF_k$ such that for every $\emptyset \neq C\subseteq B$, we have $|A\cap \Sigma_C|\leq (1-\delta)|A|$. Then, there is a set $A'$ satisfying $|A'|=61k\cdot 2^d$ with $\delta(A';\cF_k)\geq \frac{1}{8}$, where $d=\lceil{\nicefrac{1}{\delta}}\rceil$.
\end{prop}
\begin{proof}
We first spell out the main ideas: If one samples a set of size $\frac{1}{\delta}$ from $A$, then there is a constant probability of it not being contained in a specific $\Sigma_C$. By repeating this $O(k)$-times, we get a collection of sets such that for \textbf{every} syndrome $C$, a constant proportion of them are not contained in $\Sigma_C$. We then replace each of these sets by the exponential construction from Theorem \ref{thm:exp_set}. Thus, every subgroup $H$ avoids half of all the sets that were not contained in $\Sigma_{Synd(H;B)}$, which is a constant amount. 
\\

    Sample a set $Y$ of size $d=\lceil\frac{1}{\delta}\rceil$ in $A$ uniformly at random. Then, for $\emptyset \neq C\subseteq B$, we have
    \[
\Pro[Y\subseteq \Sigma_C]\leq (1-\delta)^{\nicefrac{1}{\delta}}\leq \nicefrac{1}{e}.
    \]
    Replace $Y=\{a_1,...,a_d\}$ by $\widetilde Y=\{a_S=\prod a_i^{s(a_i)}\mid S\subseteq Y\}$, as was done in Theorem \ref{thm:exp_set} for the basis $B$. If $Y\not \subseteq\Sigma_C$, then by the same argument as in Theorem \ref{thm:exp_set}, for every $H\lneq \cF_k$ with $Synd(H;B)=C$, we have that $|H\cap \widetilde Y|\leq \nicefrac{|\widetilde Y|}{2}$. 
    Let us sample independently a collection of subsets $Y(j)\subseteq A$, each of size  $d=\lceil\frac{1}{\delta}\rceil$, for $1\leq j\leq 61k$. By Lemma \ref{lem:Chernoff_bound},
    \[
        \Pro\left[|\{j\in [61k]\mid Y(j)\subseteq \Sigma_C\}|\geq \frac{122k}{e}\right]\leq e^{-\nicefrac{61k}{3e^3}}<e^{-k}.
    \]
    By applying a union bound over the collection of possible syndromes $\emptyset\neq C\subseteq B$, we get
    \[
         \Pro\left[\exists \emptyset\neq C\subseteq B \colon |\{j\in [61k]\mid Y(j)\subseteq \Sigma_C\}|\geq \frac{122k}{e}\right]<2^k e^{-k}<1,
    \]
    which in turn implies that there is a collection of $61k$ subsets $Y(j)$, each of size $d$ such that every $\Sigma_C$ contains at most a $\frac{2}{e}$-fraction of them.
    By letting $A'=\bigcup \widetilde Y(j)$, we get a set of $61k\cdot 2^d$ elements which satisfies 
    \[
    \forall H\lneq \cF_k\ \colon\ \ \frac{|A'\cap H|}{|A'|}\leq 1-\frac{1}{2}(1-\frac{2}{e})\leq 1-\frac{1}{8}.
    \]
\end{proof}
\begin{cor}\label{cor:resolving_main_1}
By taking the $A$ that was guaranteed by  Theorem \ref{thm:first_iteration}, we get a set that satisfies the conditions of Proposition \ref{prop:amplification_to_const} with $\delta=\nicefrac{1}{12\log k}$. Thus, by applying \ref{prop:amplification_to_const}  on it, we get a set $A'$ of size $61k^{13}$ with detection probability  $\nicefrac{1}{8}$. 
\end{cor}
\begin{rem}
   In the theory of error correcting codes, composition of codes is a standard technique. Corollary \ref{cor:resolving_main_1} can be seen as a non-abelian version of code composition.

   Note also that the resulting set in Corollary \ref{cor:resolving_main_1} is no longer a subset of $\widetilde B$. It consists of words of length at most $d=\lceil\frac{1}{\delta}\rceil\leq 12\log k$ in the words in $Y(j)$, which are themselves words of length at most $k$. Hence, the resulting $A'$ consists of words of length at most $12k\log k$ in the original generators $B$ of $\cF_k$.
\end{rem}

\subsection{Iterative composition}\label{sec:iterative}
We now suggest a more gradual amplification, which allows us to find smaller test subsets (which are not yet linear in $k$) with constant detection probability. The idea is similar to the code composition used in the proof of Proposition \ref{prop:amplification_to_const}, but instead of using a different code as the smaller one, we use the scaled down version of the same code. The benefit from this is exactly that the construction of Theorem \ref{thm:first_iteration} depends only on the syndromes, which will allow us to iterate it more than once and gain smaller sized $A$'s, while mildly worsening the detection probability $\delta$. 

To that end, we need to recall definitions, fix notations and observe some facts.
Let $B=\{x_1,...,x_k\}$ be a basis of $\cF_k$. Any word $w\in \cF_k$ can be written as  $w=\prod_{j=1}^\ell x_{\alpha(j)}^{\eps(j)}$,  where $\alpha\colon [\ell]\to [k]$ and $\eps\colon [\ell]\to \{\pm 1\}$. Thus, $w$ induces on every group $G$  a \emph{word map} $w\colon G^k\to G$ in the following way 
\begin{equation}\label{eq:defn_word_map}
w(g_1,...,g_k)=\prod_{j=1}^\ell g_{\alpha(j)}^{\eps(j)}.
\end{equation}
Therefore, an ordered multi-set $A=\{w_1,...,w_n\}\subseteq \cF_k$ (namely, a tuple of words) induces a \emph{set word map} $A\colon G^k\to G^n$ by mapping $\vec g=(g_1,...,g_k)$ to $(w_1(\vec g),...,w_n(\vec g))$. The detection probability of $A$ in $\cF_k$ translates to the following property: If $H\lneq G$ and $\vec g\in G^k$ are such that $\vec  g\cap H\neq \vec g$ (using the notation from Remark \ref{rem:intersection_of_tuple_and_set}), by which we mean that there are elements in the tuple $\vec g$ which are not in $H$, then $|A(\vec g)\cap H|\leq (1-\delta(A;\cF_k))|A(\vec g)|=(1-\delta)n$. 
Furthermore, for every $C\subseteq D\subseteq G$ we can denote by $\Sigma^D_C(G)$  the union  of all $H\leq G$ such that $Synd(H;D)=C$. Then, $\Sigma_C$ from before was $\Sigma^B_C(\cF_k)$ where $B$ is a basis of $\cF_k$. Now, if $A\subseteq \cF_k$ satisfied that $|A\cap\Sigma_C|\leq (1-\delta)|A|$ for every $\emptyset \neq C\subseteq B$, then $|A(\vec g)\cap \Sigma^{\vec g}_C(G)|\leq (1-\delta)|A(\vec g)|$ for every $\vec g\in G^k$ and $\emptyset\neq C\subseteq\ \vec g$.

Through this point of view, Theorem \ref{thm:first_iteration} defines a sequence of set word maps, one for each $k$, such that given an ordered subset of size $k$ in a group $G$, it outputs an ordered subset of size $432k\log k$ in $G$ with certain properties. Let us denote this map by $\Lambda$ --- note that $\Lambda$ first chceks what is the size of its input, interpret it as $k$, and then applies the $k^{\rm th}$ set word map defined by the theorem. 
Similarly, in Theorem \ref{thm:exp_set}, we constructed a set word map for each $k$ that takes ordered sets of size $k$ and produces an ordered set of size $2^k$. This process was denoted before by $Y\mapsto \widetilde Y$, but let us denote this map by $\Psi$ from now on --- again, $\Psi$ acts depending on the size of its input.
Thus, in Theorem \ref{thm:first_iteration}, Theorem \ref{thm:exp_set} and Proposition \ref{prop:amplification_to_const} we proved the following:\footnote{Note that we use $\log|A|$ as if it is an integer. We actually mean $\lceil \log|A|\rceil$, but omit the rounding from the notation.} 
\begin{cor}\label{cor:synopsis}
    For every ordered subset $\vec g$ of a group $G$ ---
\begin{enumerate}
    \item $|\Lambda(\vec g)|=432|\vec g |\log|\vec g |$.
    \item If $\emptyset \neq C\subseteq \vec g $, then $|\Sigma_C^{\vec g}(G) \cap \Lambda(\vec g )|\leq (1-\frac{1}{12\log|\vec g |})|\Lambda(\vec g )|$.
    \item There is a collection $\{Y(j)\}$ of $61|\vec g |$ subsets of $\Lambda(\vec g )$ of size $12\log|\vec g |$ each, satisfying
    \[
\forall \emptyset \neq C\subseteq \vec g \ \colon \ \ \Pro_{j}[Y(j)\subseteq \Sigma_C^{\vec g }(G)]\leq \nicefrac{2}{e}.
    \]
    \item $|\Psi(\vec g )|=2^{|\vec g |}$.
    \item If $Synd(H;\vec g )\neq \emptyset$, then $|H\cap \Psi(\vec g )|\leq \frac{|\Psi(\vec g )|}{2}$.
    \item Denote by $L(\vec g )=\max\{\ell(w)\mid w\in \vec g \}$ the maximal length of a word in $\vec g $ with respect to some prefixed generating set of $G$. Then, $$L(\Psi(\vec g )),L(\Lambda(\vec g ))\leq |\vec g |L(\vec g ).$$
\end{enumerate}
\end{cor}
Corollary \ref{cor:synopsis} defines a map $\vec g \mapsto \{Y(j)\mid j\leq 61|\vec g |\}$. Applying this operation  on each $Y(j)$ outputs  a collection $\{Y(j,j')\mid j\leq 61|\vec g |,j'\leq 61|Y(j)|=61\cdot 12\log|\vec g |\}$. Applying this operation again on each $Y(j,j')$ provides  sets $Y(j,j',j'')$ and we can keep going.  After $t$ steps, we have sets labeled by $Y(j,j',...,j^{(t-1)})$. Each of these sets is of size \[
\underbrace{12\log(12\log(12\log...12(\log |\vec g |)...))}_{t-\textrm{times}},\]
and there are
\[
61|\vec g |\times (61\cdot12\log |\vec g |)\times (61\cdot 12\log(12\log |\vec g |))...=61^t\cdot 12^{t-1}\cdot\prod_{j=0}^{t-1}\underbrace{\log(12\log(12\log...12(\log |\vec g |)...))}_{j-\textrm{times}}
\]
many of them.
By construction, if $\vec g$ is a generating set of $G$, then for every  $H\lneq G$ its syndrome in at least $(1-\frac{2}{e})^t$ of the $Y(j,j',...,j^{(t-1)})$ is non-trivial. Hence, if we construct $\Psi(Y(j,j',...,j^{(t-1)}))$, then $H$ avoids $\frac{1}{2}\cdot(1-\frac{2}{e})^t$ of the words in $\cA=\bigcup \Psi(Y(j,j',...,j^{(t-1)}))$, while 
\[
|\cA|=\left(61^t\cdot 12^{t-2}\prod_{j=0}^{t-2}\underbrace{\log(12\log(12\log...12(\log |\vec g |)...))}_{j-\textrm{times}}\right)\times 2^{\overbrace{12\log(12\log(12\log...12(\log |\vec g |)...))}^{t-\textrm{times}}}.
\]
Now, if we begin with $\vec g =B$ the basis of the free group $\cF_k$, then $|\vec g |=k$ and $$|\cA|=O_t(k\cdot\log k\cdot\log\log k \cdot...\cdot(\underbrace{\log\log...\log k}_{t-1\ \textrm{times}})^{13}).$$
The construction used in Corollary \ref{cor:resolving_main_1} is exactly this one with $t=1$. If we apply this construction with   bigger $t$'s, then the $\delta$ parameter deteriorates exponentially in $t$, while the sets (asymptotically) shrink by a factor of $(\overbrace{\log\log...\log 
 k}^{{t-1\ \textrm{times}}}/\underbrace{\log\log...\log k}_{t-2\ \textrm{times}})^{13}$. 
 \begin{cor}\label{cor:two_iterations}
     By applying the operation above with $t=2$, the obtained set $\cA$ is of size $O(k\log^{13} k)$ and has a constant detection probability \nicefrac{1}{32}, proving Theorem \ref{thm:intro_poly_sized_set}. In this case, $L(A)$ is $O(k\log k \log\log k)$.
 \end{cor}

 \subsection{Good universal codes via iterative encoding}\label{sec:Spielman}
In this section we prove our main result, Theorem \ref{thm:Spielman_construction} (which implies Theorem \ref{conj:main}). To that end, we use Spielman's `simple' construction of error-reduction codes from \cite[Section 4]{spielman1995linear} as a blueprint. 
In the spirit of the partial solution from the previous section, we view constructions of sets with high detection probability as functions from ordered subsets of the free group to ordered subsets of the free group via their set word map \eqref{eq:defn_word_map}.

Let us first use a bipartite graph to define a tuple of words, and thus a set word map. Given  $\Gamma=(L,R,E)$ with $L=\{x_1,...,x_k\}$ and $R=\{b_1,...,b_n\}$ being ordered sets, we can associate with each vertex $b$ on the right side of $\Gamma$ a word $w=w_b$ as follows: If $N(b)$ is the collection of neighbors of $b$ in $L$, then we define $w_b=\prod_{x\in N(b)}x$, where the ordering of the product is induced by the order on $L$ itself. This means that every bipartite graph defines a set word map $\Upsilon_\Gamma\colon G^k\to G^n$, induced by the words $w_b$, on every group $G$.

 \begin{lem}[Existence of unique neighbor expanders]\label{lem:unique_neighbors1}
     For every  rational $0<\beta\leq 1$ and $0<\eps<\nicefrac{1}{2}$, there are  large enough $d,n_0\in \mathbb{N}$ and a small enough $\alpha>0$, such that for every $n\geq n_0$ such that $\beta n$ is natural, there is a bipartite graph with $n$ vertices on the left side $L$, $\beta n$ vertices on the right side $R$, which is left $d$-regular, and such that  every set $S\subseteq L$ of size at most $\alpha n$ has at least $(1-\eps)d|S|$ many neighbors in $R$. In particular, every such $S$ has at least $(1-2\eps)d|S|$ many unique neighbors.

     We call such graphs $(d,\beta,\alpha,1-2\eps)$ unique neighbor expanders.\footnote{They are usually called $(d,\alpha,\eps)$-lossless expanders. But, as we want to emphasize the size of the unique neighbors set and the relative sizes of the two sides of the graph, we chose our terminology.} 
 \end{lem}
\begin{rem}\label{rem:explicit_unique_neighbor_exist}
Explicit constructions that prove the above Lemma appear in \cite{capalbo2002randomness,alon2002explicit,golowich2024new}.
  We provide a non-constructive proof of Lemma \ref{lem:unique_neighbors1} in Appendix \ref{sec:existence_lossless}. This proof is \textbf{very} standard, and versions of it can be traced back to \cite{pinsker1973complexity} and even \cite{kolmogorov1967realization}. We used  the proof appearing in  Lemma 1.9 of \cite{Hoory_Linial_Wigderson} as a blueprint. A slightly more involved (yet still non-constructive) proof, with an essentially optimal $\eps$ parameter, can be found in Appendix II of \cite{sipser1996expander}. 
\end{rem}

\begin{rem}\label{rem:parameters_unique_neighbors_spielman}
 In this section, we use unique neighbor expanders with $\beta=\nicefrac{1}{2}$ and $\eps=\nicefrac{1}{4}$. By Lemma \ref{lem:unique_neighbors1}, there are (explicit computable) constants $d,n_0,\alpha$ such that for every even $n$ there is a (explicit) $(d,\nicefrac{1}{2},\alpha,\nicefrac{1}{2})$-unique neighbor expander. Specifically, we can assume $d\geq 16$. 
 
 The reader who wants concrete numbers and is willing to give up explicitness of the construction, can use the proof of the above Lemma  appearing in Appendix \ref{sec:existence_lossless},  and take $d=16$, $\alpha=2^{-128}$ and $n_0= 2^{128}$.  
 
Fix for every sufficiently large even $n$  a $(d,\nicefrac{1}{2},\alpha,\nicefrac{1}{2})$-unique neighbor expander and denote it  by $\Gamma_n$. 
   
\end{rem}

 \begin{lem}\label{lem:iterations_of_Spielman}
      Assume that there is a subset $A\subseteq \cF_k$ of size $4k$, with $\delta(A;\cF_k)\geq\delta$ for $0<\delta\leq \nicefrac{\alpha}{4}$ and $k\geq n_0$. Then, there is a subset $A'\subseteq\cF_{2k}$ of size $8k$ with $\delta(A';\cF_k)\geq \delta$.
 \end{lem}
\begin{proof}
    Let $\Gamma_{2k}$ and $\Gamma_{4k}$ be the $(d,\nicefrac{1}{2},\alpha,\nicefrac{1}{2})$-unique neighbor expanders of sizes $2k$ and $4k$ respectively promised by Remark \ref{rem:parameters_unique_neighbors_spielman}. Define the following set $A'$. Start with the standard basis $B$ of $\cF_{2k}$. Apply the set word map $\Upsilon_{\Gamma_{2k}}$ on $B$ to get $D$, and note that $|D|=k$. Then, apply the set word map $A$ on $D$, to get $E$, and note that $|E|=4k$. Finally, apply the set word map $\Upsilon_{\Gamma_{4k}}$ on $E$ to get $F$, and note that $|F|=2k$. For $A'$, take the union (of ordered multi-sets) $B\cup E\cup F$, and note that $|A'|=8k$. See Figure \ref{fig:Spielman} for visualisation.

    Let $H\lneq \cF_{2k}$ be a proper subgroup. If $|Synd(H;B)|\geq 8\delta k$, then since $A'$ contains $B$ we deduce that $1-\nicefrac{|H\cap A'|}{8k}\geq \delta$. Otherwise, as $8\delta k\leq \alpha\cdot 2k$, by the fact $\Gamma_{2k}$ is a unique neighbor expander with $1-2\eps=\nicefrac{1}{2}$, there are at least $|Synd(H;B)|\cdot\nicefrac{d}{2}>0$ many words in $D=\Upsilon_{\Gamma_{2k}}(B)$ that contain exactly one of the elements of $Synd(H;B)$ in their defining product, and thus are not in $H$ --- this is our recurring argument, see Observation \ref{obs:one_trick} or   \eqref{eq:syndrome_implies_notin}. By the assumption on $A$, if $Synd(H;D)=Synd(H\cap\langle D\rangle ; D)\neq \emptyset$, then 
    $$|Synd(H;A(D))|=|Synd(H\cap \langle D\rangle ;A(D))|\geq \delta |D|=\delta k,$$ and $A(D)=E\subseteq A'$. If $|Synd(H;E)|\geq 8\delta k$, then we are done as before. 
    Otherwise, $$\delta k\leq |Synd(H;E)|\leq 8\delta k\leq \alpha\cdot 4k.$$ 
    So, since $\Gamma_{4k}$ is a unique neighbor expander with $1-2\eps=\nicefrac{1}{2}$, there are at least $\delta k\cdot \nicefrac{d}{2}$ words in $F=\Upsilon_{\Gamma_2}(E)$ with exactly one element out of $H$ in the product defining them, and they are thus not in $H$ (Observation \ref{obs:one_trick}). Namely, $|Synd(H;F)|\geq \delta k\cdot \nicefrac{d}{2}\geq 8\delta k$ in this case, as $d\geq 16$. All in all, in every possible case we deduced that $Synd(H;A')\geq 8\delta k$, and as $|A'|=8k$, we get $\delta(A';\cF_k)\geq \delta$ as needed.
\end{proof}

\begin{figure}
	\centering
	\begin{tikzpicture}[scale=1]
	\node[color=black] (B) at (-1,7) {The set $B$};
	\node[color=black] (D) at (3,7) {The set $D$};
	
	\node[color=black] (E) at (9,7) {The set $E$};
	
	\node[color=black] (F) at (13,7) {The set $F$};

	\node[draw, color=black, shape=circle] (x1) at (-1,5) { $\ \ x_{1}\ \ $};
	
	\node[draw, color=black, shape=circle] (x2) at (-1,2.5) { $\ \ x_2\ \ $};
	
	\node (xdots) at (-1,0) { $\vdots$};
	
	\node[draw, color=black, shape=circle] (x2k-1) at (-1,-2.5) { $x_{2k-1}$};
	
	\node[draw, color=black, shape=circle] (x2k) at (-1,-5) { $\ x_{2k}\ $};

	\node[draw, color=black, shape=diamond] (w1) at (3,4) { $\ w_1\ $};
	
	\node[draw, color=black, shape=diamond] (w2) at (3,2) { $\ w_2\ $};
	
	\node (wdots) at (3,0) { $\vdots$};
	\node[draw, color=black, shape=diamond] (wk-1) at (3,-2) { $w_{k-1}$};
	\node[draw, color=black, shape=diamond] (wk) at (3,-4) { $\ w_{k}\ $};

	\node[color=cyan] (graph2k) at (1.2,-7) { The graph $\Gamma_{2k}$};

	\node[draw, color=black, shape=circle] (v1) at (9,6) { $\ \ v_1\ \ $};
	
	\node[draw, color=black, shape=circle] (v2) at (9,3) { $\ \ v_2\ \ $};
	
	\node (vdots) at (9,0) { $\vdots$};
	\node[draw, color=black, shape=circle] (v4k-1) at (9,-3) { $v_{4k-1}$};
	\node[draw, color=black, shape=circle] (v4k) at (9,-6) { $\ v_{4k}\ $};
	
	\node[color=magenta] (Awordmap) at (6,0.5) { The set word map $A$};

	\node[draw, color=black, shape=circle] (y1) at (13,4) { $\ \ y_1\ \ $};
	
	\node[draw, color=black, shape=circle] (y2) at (13,2) { $\ \ y_2\ \ $};
	
	\node (ydots) at (13,0) { $\vdots$};
	
	\node[draw, color=black, shape=circle] (y2k-1) at (13,-2) { $y_{2k-1}$};
	
	\node[draw, color=black, shape=circle] (y2k) at (13,-4) { $\ y_{2k}\ $};
	\node[color=cyan] (graph4k) at (11.3,-7) { The graph $\Gamma_{4k}$};

	%edges pauli X
	\draw[cyan, -, dashed] (x1)--(w2);
	\draw[cyan, -, dashed] (x2)--(wk-1);
	\draw[cyan, -, dashed] (x1)--(wdots);
	\draw[cyan, -, dashed] (x1)--(w1);
	\draw[cyan, -, dashed] (x2)--(wk-1);
	\draw[cyan, -, dashed] (x2)--(w2);
	\draw[cyan, -, dashed] (x2)--(wdots);
	\draw[cyan, -, dashed] (x2k-1)--(wdots);
	\draw[cyan, -, dashed] (x2k-1)--(w1);
	\draw[cyan, -, dashed] (x2k-1)--(wk);
	\draw[cyan, -, dashed] (x2k)--(wdots);
	\draw[cyan, -, dashed] (x2k)--(w1);
	\draw[cyan, -, dashed] (x2k)--(w2);
	
	\draw[magenta, ->, solid] (wdots)--(vdots);
	
	\draw[cyan, -, dashed] (v1)--(y2k);
	\draw[cyan, -, dashed] (v2)--(y2k-1);
	\draw[cyan, -, dashed] (v1)--(ydots);
	\draw[cyan, -, dashed] (v1)--(y1);
	\draw[cyan, -, dashed] (v2)--(y2);
	\draw[cyan, -, dashed] (v2)--(ydots);
	\draw[cyan, -, dashed] (v4k-1)--(ydots);
	\draw[cyan, -, dashed] (v4k-1)--(y1);
	\draw[cyan, -, dashed] (v4k-1)--(y2k);
	\draw[cyan, -, dashed] (v4k)--(ydots);
	\draw[cyan, -, dashed] (v4k)--(y2k);
	\draw[cyan, -, dashed] (v4k)--(y2k-1);
	\end{tikzpicture}
	
	\caption{This is a visualisation of $A'$ that is constructed in Lemma \ref{lem:iterations_of_Spielman}.  The set $B=\{x_1,...,x_{2k}\}$ is the standard basis of $\cF_{2k}$. The $w_i$'s play both the role of right side vertices in $\Gamma_{2k}$ as well as the words $w_i(B)$ induced by the set word map $\Upsilon_{\Gamma_{2k}}$, which constitute the set $D$. The $v_i$'s are the image of the set word map $A$ applied on the tuple $D=(w_i(B))_{i=1}^{k}$, which constitutes $E$. Finally, the set $E=(v_i)_{i=1}^{4k}$ plays the role of the left side of $\Gamma_{4k}$ as well, and the $y_i$'s are the right hand side as well as the images of $\Upsilon_{\Gamma_{4k}}$, namely $F=(y_i(E))_{i=1}^{2k}$. All in all, $A'=B\cup E\cup F=\{x_1,...,x_{2k},v_1,...,v_{4k},y_1,...,y_{2k}\}$.
		Note that $\Gamma_{2k}$ and $\Gamma_{4k}$ are supposed to be $d$-left regular, as oppose to the Figure which only illustrates $3$-regularity.}

\end{figure}
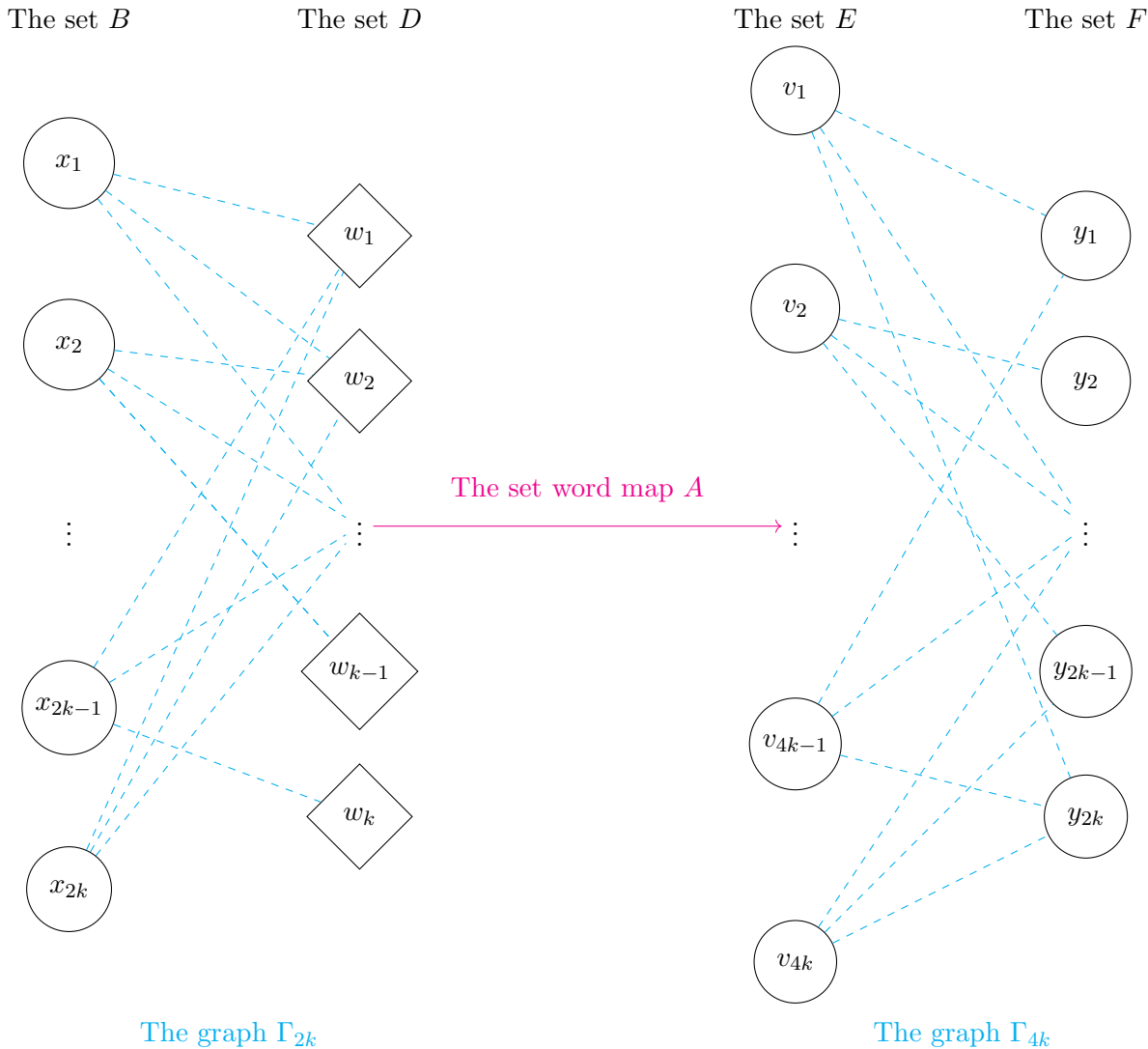

\begin{cor}\label{cor:blabla}
    Good universal codes exist. 
\end{cor}
\begin{proof}
Choose $k_0$ to be the smallest positive integer such that  $\delta=\nicefrac{1}{k_0}\leq \min\{\nicefrac{1}{n_0},\nicefrac{\alpha}{4}\}$. Take $A_0$ to be $4$ copies of the basis of $\cF_{k_0}$. Then, $A_0$ satisfies the conditions of Lemma \ref{lem:iterations_of_Spielman}, as its detection probability is $\nicefrac{1}{k_0}=\delta \leq \nicefrac{\alpha}{4}$ and $k_0\geq n_0$. Also the output $A'$ of Lemma \ref{lem:iterations_of_Spielman} satisfies its conditions given that the input satisfied them.  Hence, applying the Lemma repeatedly on $A_0$ proves the corollary.
\end{proof}

Combining Corollary \ref{cor:blabla} with item $(2)$ of Remark \ref{rem:density}, resolves Theorem \ref{thm:Spielman_construction} and thus Theorem \ref{conj:main}. Note that in Theorem \ref{thm:Spielman_construction} we used the bound $|A|\leq 8k$ while Corollary \ref{cor:blabla} provides $|A|=4k$. This is because of the density issue --- Corollary \ref{cor:blabla} works only for $k$'s which are large enough powers of $2$ times $k_0$ --- and the appropriate application of   item $(2)$ of Remark \ref{rem:density} results in a multiplicative factor $2$ loss. Note that the lengths of the words in $A'$ are  $DD'\cdot L(A)$, where $D$ is the maximal right degree of $\Gamma_{2k}$ and $D'$ is the maximal right degree of $\Gamma_{4k}$. By choosing unique neighbor expanders that are bi-regular (e.g., those from  \cite[Apendix II]{sipser1996expander}), the right regularity is $\nicefrac{d}{\beta}=2d$, which means the length of this construction is polynomial in $k$ with exponent $\log(4d^2)$.

\section{Abelian codes}\label{sec:abelian}
As discussed in the introduction, we can study $\frak{G}$-codes for more restricted classes of groups $\frak{G}$.  
In this section, we prove Theorem \ref{thm:abelian_case}, 
namely find a test subset of linear size with constant detection probability in   \emph{free abelian groups}, and hence for every abelian group. 
Viewing this problem from the point of view of error-correcting codes, we are looking for a single generating matrix with integer coefficients, such that it is a good code $\text{mod }m$ simultaneously for all $m\in \mathbb{N}$. In particular, we construct one generating matrix that works simultaneously over all finite fields $\mathbb{F}_p$. 

Similarly to the previous section, here as well the methods are inspired by  error correcting codes. This time, these are Tanner codes with respect to unique neighbor expanders. As explicit constructions of unique neigbor expanders exist (Remark \ref{rem:explicit_unique_neighbor_exist}), it allows us to provide an explicit construction. 

Let us recall what we are looking for:
There exist constants $C\geq 1$ and $\delta>0$, such that for every $k\in \mathbb{N}$, there
is a subset $A\subseteq\mathbb{Z}^{k}$ satisfying 
\begin{align}
     &|A|\leq Ck;\label{clause1:thm_abelian}\\
    \forall H\lneq \mathbb{Z}^k\ \colon \ \ &\frac{\left|A\cap H\right|}{|A|}\leq1-\delta.\label{clause2:thm_abelian} 
\end{align}
\begin{rem}
    In  Theorem \ref{thm:abelian_case} we stated that we can find $A$ such that $C\leq 4$ and $\delta\geq \nicefrac{1}{320}$. But, in this Section we describe a general way of constructing these $A$'s, and the exact parameters are postponed to the end.
\end{rem}

Let $\cE$ be an $n\times k$ matrix with integer coefficients. Then, it defines a homomorphism $\enc \colon \mathbb{Z}^k\to \mathbb{Z}^n$, and, as in Section \ref{sec:proper_subgps_and_ECCs}, we are interested in the image $\cC(\cE)=\cC=\Img (\enc)$ which is a subgroup of $\mathbb{Z}^n$. Note that $\cC$ is the subgroup generated by the columns of $\cE$. Furthermore, 
\begin{equation}\label{eq:enc_by_E}
    \enc(v)=\cE\cdot v=(\langle a,v\rangle)_{a\in \textrm{Rows}(\cE)}.
\end{equation}
We abuse notation and keep using $\Phi_p$ as the $\mod p$ operation on all coordinates. Then $\Phi_p(\cC)$ is a linear subspace of $\FF_p^n$, and we can study its properties as an error correcting code. 
To prove Theorem \ref{thm:abelian_case}, we first reduce it to the following proposition.

\begin{prop}\label{prop:matrices_to_abelian_case}
There exist constants $C\geq 1$ and $\delta>0$, such that for every $k\in \mathbb{N}$, there is a matrix $\mathcal{E}$  with integer coefficients
of size $n\times k$ satisfying
the following:
\begin{enumerate}
    \item $n\leq Ck$.
    \item For every prime number 
$p$, the linear subspace $\Phi_p(\cC(\cE))$ is an $[n,k,\delta n]_p$-code.
\end{enumerate}
\end{prop}
\begin{rem}
    By applying the methods of  Section \ref{sec:ECC}, a solution to Theorem \ref{thm:abelian_case} automatically implies a solution to Proposition \ref{prop:matrices_to_abelian_case}. The main take away from the following proof is that resolving Proposition \ref{prop:matrices_to_abelian_case} is also sufficient with the same constants $C$ and $\delta$. Also, as in Remark \ref{rem:simultanious2}, this construction works uniformly to all prime fields. But, as opposed to our solution in Section \ref{sec:fin_abelian_gps}, it is \emph{the same} encoding matrix regardless of the alphabet. 
\end{rem}

\begin{proof}[Proof of Theorem \ref{thm:abelian_case} assuming Proposition \ref{prop:matrices_to_abelian_case}]
Let $C_{0}\geq 1$, $\delta_{0}>0$ be the constants, and $\mathcal{E}$
the matrix of size $n\times k$, guaranteed by Proposition
\ref{prop:matrices_to_abelian_case}. Let $A=\textrm{Rows}(\cE_k)$ be the rows of the matrix $\mathcal{E}$. Since
$|A|=n\leq C_0k,$  clause \eqref{clause1:thm_abelian} is satisfied with $C=C_{0}$. 

Now, let us prove that clause \eqref{clause2:thm_abelian} is satisfied. As before,
we can assume $H$ is a maximal proper subgroup of $\mathbb{Z}^{k}.$
Since every subgroup of $\mathbb{Z}^{k}$ is normal, $H$ is the kernel
of a non-trivial homomorphism $f\colon\mathbb{Z}^{k}\to G$, where
$G$ is a simple group. Since $G$ is also the image of an abelian
group, it must be cyclic of prime order. Namely, $f\colon\mathbb{Z}^{k}\to\mathbb{F}_{p}$
for some prime number $p$. For every such $f$, $\ker \Phi_p=\left(p\mathbb{Z}\right)^{k}$
is contained in ${\rm ker}f=H$, which in turn implies that $f$ factors through $\FF_p^k$. Hence, there is a correspondence
between  maximal subgroups of $\mathbb{Z}^{k}$  and non-trivial functionals $\varphi\colon\mathbb{F}_{p}^{k}\to\mathbb{F}_{p}$, when $p$ runs over all primes.\footnote{The correspondence is
not one to one, since $\varphi$ and $\beta\varphi$ have the same
kernel for $0\neq\beta\in\mathbb{F}_{p}$, but it does not effect
our argument.} Let $\vec{0}\neq\alpha\in\mathbb{F}_{p}^{k}$ be the vector, as in Section \ref{sec:proper_subgps_and_ECCs}, for which $\varphi=\varphi_{\alpha}$, i.e., $\varphi_{\alpha}(v)=\left\langle v,\alpha\right\rangle =\sum_{i=1}^{k}v_{i}\alpha_{i}$.
Then
\begin{align*}
A\cap H & =\left\{ a\in A\mid f(a)=0\right\} \\
 & =\left\{ a\in \Phi_p(A)\mid\varphi_{\alpha}(a)=0\right\} \\
 & =\left\{ a\in \textrm{Rows}(\cE)\mid\left\langle a,\alpha\right\rangle =0\right\}.
\end{align*}
By \eqref{eq:enc_by_E}, $\left\{ a\in \textrm{Rows}(\cE)\mid\left\langle a,\alpha\right\rangle =0\right\}$ is the number of coordinates in $\enc(\alpha)=\cE\cdot \alpha$ that are zero. 
But,  clause $(2)$ of Proposition \ref{prop:matrices_to_abelian_case} tells us that $\Img(\enc)=\Phi_p(\cC(\cE))$ has dimension $k$ and distance $\delta_0 n$. The fact that $\Img(\enc)$ has dimension $k$ implies that $\enc$ is injective, and since $\alpha\neq \vec 0$, then also $ \vec 0\neq\enc(\alpha)\in \Img(\enc)$. Hence,
\[
|A\cap H|=n-w_H(\enc(\alpha))\leq n(1-\delta_0), 
\]
which proves clause \eqref{clause2:thm_abelian}  with $\delta=\delta_0$.
\end{proof}

In the rest of this section we discuss a way to construct a collection  of matrices satisfying
the conditions of Proposition \ref{prop:matrices_to_abelian_case}, and hence proving it, which implies a proof for Theorem \ref{thm:abelian_case}.

Instead of viewing a code as the image of an encoding map, one can view it as the kernel of a linear map. I.e., if $\cC\subseteq \FF_p^n$ is a linear subspace, there is a map $\pi\colon \FF_p^n\to \FF_p^m$ whose kernel is $\cC$. The matrix associated with $\pi$ is usually referred to as a \emph{parity check matrix} for $\cC$. 
Now, one can take the same approach to generate subgroups of $\mathbb{Z}^n$. Not every subgroup is the kernel of a map $\pi\colon \mathbb{Z}^n\to \mathbb{Z}^m$, but kernels of this form are  still a good source for subgroups. Furthermore, these subgroups have some helpful extra properties.

\begin{claim}\label{claim:prop_of_ker_of_Z_homs}
    Let $H$ be the kernel of a map $\pi\colon \mathbb{Z}^n\to \mathbb{Z}^m$. Then,
    \begin{enumerate}
        \item The rank of $H$ is at least $n-m$. 
        \item Every basis of $H$ stays linearly independent when taken $\mod p$.
    \end{enumerate}
\end{claim}

\begin{proof}
    Let us view $\pi$  as an integral $m\times n$ matrix. Then, the $\mathbb{Z}$-kernel of $\pi$, which is $H$, is the intersection of $\mathbb{Z}^n$ with the $\mathbb{Q}$-kernel of $\pi$. Moreover, the intersection of every $d$-dimensional subspace of $\mathbb{Q}^n$
with $\mathbb{Z}^n$ must have rank $d$. Hence, $\textrm{rank}(H)=\dim_\mathbb{Q}(\Ker(\pi))\geq n-m$. Furthermore, since $\mathbb{Z}^n/H$ is a subgroup of $\mathbb{Z}^m$, it has no torsion, and hence $\nicefrac{\mathbb{Z}}{H}$ is a free abelian group. Thus, $H$ is a free factor of $\mathbb{Z}^n$. In particular, every basis of $H$ can be completed to a basis of $\mathbb{Z}^n$, which is linearly independent modulo every prime number $p$. This proves clause $(2)$.
\end{proof}

A standard way for choosing codes is to use the adjacency matrix of a bipartite graph as the parity matrix of the code.\footnote{The bipartite graph which represents the constraints  is usually referred to as a \emph{Tanner graph} for the code.} A bipartite graph $\Gamma=(L,R,E)$ defines for every abelian group $G$ a homomorphism\footnote{Note that the homomorphism $\pi_G$ is exactly the set word map $\Upsilon_G$ used in Section \ref{sec:Spielman} in the abelian case, which turns out to induce a homomoephism in the class of abelian groups.} $\pi=\pi_G\colon G^L\to G^R$ by 
\begin{equation}\label{eq:bipartite_defines_map}
    \forall f\colon L\to G\ ,\  w\in R\ \colon \ \ \pi f(w)=\sum_{v\sim w}f(w).
\end{equation}
When $G=\FF_p$, the kernel of $\pi$ is a code.
Recall that a bipartite graph $\Gamma=(L,R,E)$ is an \emph{$\alpha$-left unique neighbor
expander} if for every set $S\subseteq L$ satisfying $|S|\leq\alpha|L|$,
one has a vertex $w\in R$ for which $|E\cap\left(S\times\{w\}\right)|=1,$
namely $v$ is connected by exactly one edge to $S$.
In words, every small enough subset of the left side of the
graph has at least one vertex on the right side of the graph that
touches it exactly once.\footnote{As opposed to Section \ref{sec:Spielman}, here we need this much weaker notion of unique neighbor expansion. The construction used therein, usually referred to as lossless expander, is not needed in this section.}

\begin{claim}\label{claim:unique_neighbor_implies_distance}
Let $0<\alpha<1$.
Let $\Gamma=\left(L,R,E\right)$ be an $\alpha$-left unique neighbor expander. Let $\pi_\mathbb{Z}\colon \mathbb{Z}^L\to \mathbb{Z}^R$ be the homomorphism defined in \eqref{eq:bipartite_defines_map}, and let $H$ be the kernel of $\pi_\mathbb{Z}$. Then, for every prime number $p$ and every $v\in H$, either $\Phi_p(v)=\vec 0$ or $w_H(\Phi_p(v))\geq \alpha |L|$.
\end{claim}

\begin{proof}
    Let us observe the following commutative diagram:
    \begin{center}
         \begin{tikzcd}
\mathbb{Z}^L \arrow[d, "\Phi_p"] \arrow[r, "\pi_\mathbb{Z}"]& \mathbb{Z}^R \arrow[d, "\Phi_p"] \\
  \FF_p^L  \arrow[r,"\pi_{\FF_p}"]& \FF_p^R
\end{tikzcd}
    \end{center}
    From it, we can deduce that $\Phi_p(H)=\Phi_p(\Ker(\pi_\mathbb{Z}))\subseteq \Ker(\pi_{\FF_p})$. Let $f\colon L\to \FF_p$ be a function, and let $S\subseteq L$ be the support of $f$. Assume $0<|S|=w_H(f)\leq \alpha|L|$. Thus, by the fact $\Gamma$ is an $\alpha$-left unique neighbor expander, there exist a $w\in R$ and $v\in S$ such that $v\sim w$ and no other vertex in $S$ is a neighbor of $w$. Hence, 
    \[
\pi_{\FF_p}f(w)=\sum_{v'\sim w}f(v')=f(v)+\sum_{v\neq v'\sim w}f(\underbrace{v'}_{\notin S})=f(\underbrace{v}_{\in S})\neq 0,
    \]
    and $f\notin \Ker(\pi_{\FF_p})$. Namely, every non-zero vector in $\Ker(\pi_{\FF_p})$ has Hamming weight of at least $\alpha|L|$. In particular, for every $v\in H$, $\Phi_p(v)$ is either $\vec 0$ or it is a non-zero vector in $\Ker(\pi_{\FF_p})$, which we just proved must satisfy $w_H(\Phi_p(v))\geq \alpha |L|$. 
   
\end{proof}

There are various ways to get bipartite $\alpha$-unique neighbor expanders $\Gamma=(L,R,E)$, as Remark \ref{rem:explicit_unique_neighbor_exist} demonstrates. Some of them  are non-explicit --- similar to the proof appearing in Appendix \ref{sec:existence_lossless}, which uses  a probabilistic-method argument --- and some are explicit. In both cases, there is a way of extracting the constant $\alpha$ and an upper bound $\beta$ on the ratio $\nicefrac{|R|}{|L|}$.

%Namely:
%\begin{fact}\label{fact:unique_neighbor_exp_exist}
%There exist constants $0<\alpha,\beta<1$ and an infinite family of bipartite
%graphs $\left\{ \Gamma=\left(L,R,E\right)\right\}$
%such that $|R|\leq \beta|L|$ and all of them are $\alpha$-left unique
%neighbor expanders. 
%\end{fact}

%\begin{rem}
 %   The random method in \cite{Hoory_Linial_Wigderson} generate a family of graphs whose left sizes $|L|$ covers all but finitely many  natural numbers. In the explicit constructions in \cites{capalbo2002randomness,alon2002explicit}, the collection of possible sizes for $|L|$ is sparser. Yet, it is still dense enough to use clause $(2)$ of Remark \ref{rem:maximal_subgroups_suffice}. 
%\end{rem}

\begin{proof}
[Proof of Proposition \ref{prop:matrices_to_abelian_case}] Let $\{\Gamma=(L,R,E)\}$ be the infinite collection of $\alpha$-left unique neighbor expanders with $|R|= \beta|L|$ guaranteed by Lemma \ref{lem:unique_neighbors1}. For every $\Gamma$ in the collection, let $\pi_\mathbb{Z}\colon \mathbb{Z}^L\to \mathbb{Z}^R$ be the homomorphism described in \eqref{eq:bipartite_defines_map}, and $H=\Ker(\pi_\mathbb{Z})$. Denote  $k=\textrm{rank}(H)$ and $n=|L|$. By Claim \ref{claim:prop_of_ker_of_Z_homs}, 
\[
k=\textrm{rank}(H)\geq |L|-|R|= (1-\beta)|L|=(1-\beta)n.
\]
Take any basis of $H$ and put it as columns of a matrix $\cE$. Then $\cC(\cE)=H$, and $\cE$ is of size $n\times k$. By Claim \ref{claim:prop_of_ker_of_Z_homs}, the columns of $\Phi_p(\cE)$ are linearly independent, which means that $\Phi_p(\cC(\cE))$ has dimension $k$ regardless of $p$. By Claim \ref{claim:unique_neighbor_implies_distance}, $\Phi_p(\cC(\cE))$ has distance of at least $\alpha|L|=\alpha n$. By choosing $C=\frac{1}{1-\beta}$ and $\delta=\alpha$, the matrix $\cE$ satisfies the conditions of Proposition \ref{prop:matrices_to_abelian_case} and we are done.
\end{proof}

Regarding lengths, by using a $d$-regular $\alpha$-unique neighbor expander, one can show --- for example, using Edmond's version of Gaussian elimination \cite[Theorem 2]{edmonds1967systems} --- that there is a basis for $\ker \pi_{\mathbb{Z}}$ such that all its entries are bounded by $2^{O(k)}$, which in turn bounds the length by a similar bound.

\begin{cor}\label{cor:nilpotent}
    Theorem \ref{thm:abelian_case} holds also for nilpotent groups, with the same constants. 
\end{cor}

\begin{proof}
    Every maximal subgroup of a finitely generated nilpotent group contains the commutator subgroup, and hence clause $(3)$ of Fact \ref{fact:basic_properties} finishes the proof.
\end{proof}

\section{Profinite groups with polynomial maximal subgroup growth}\label{sec:PMSG}

Let us start with a few sentences about profinite groups, for more see \cites{wilson1998profinite,ribes1970introduction,ribes2000profinite,fried2005field}. A profinite group $G$ is a topological group which is Hausdorff, compact and totally disconnected. Equivalently, $G$ is isomorphic (as a topological group) to an inverse limit of finite groups. A topological group $G$ is topologically generated by a set $A$, if the subgroup generated by $A$ is dense in $G$. The topological rank of $G$ in this case is the size of its smallest topological generating set. We keep using the notation $\textrm{rank}(G)$ for the topological rank if $G$ is topological. 
Since a profinite group $G$ is compact, it is equipped with a unique Haar measure $\mu$ satisfying $\mu(G)=1$. It therefore follows that an open, an thus finite index, subgroup $H$ of index $n$ has Haar measure $\mu(H)=\frac{1}{n}$.

In what follows we will study the proper subgroup testing problem for profinite groups. From that, we will deduce results about certain collections of finite groups. We use the notation $\delta(A,G)$ as before, quantifying only over proper \emph{closed} subgroups $H$. Unlike discrete groups, in the category of profinite groups, every proper closed subgroup is contained in a proper open subgroup, and we can deduce that 
\[
   \delta(A;G)=\inf\left\{1-\frac{|A\cap H|}{|A|}\ \middle|\ H\lneq_{open} G\right\}.
\] 
We mention in passing that for a (topologically) finitely generated profinite group $G$, a deep result of Nikolov--Segal \cite{nikolov2007finitely} asserts that every finite index subgroup of $G$ is open, but we will not use this result.

 A (profinite) group $G$ is said to have {\emph{polynomial maximal subgroup growth}} (PMSG) if there exists $E = E(G) >0$ such that $m_N(G) \le N^E$ for every $N \in \mathbb{N}$, where $m_N(G)$ is the number of maximal open subgroups of index $N$ in $G$.
A family of groups $ \frak{G}$ is said to have {\emph{uniform PMSG}} if there exists $E'>0$, such that for every $G \in \frak{G}$, $m_N(G) \le N^{E'\cdot \textrm{rank}(G)}$.
  In \cite{mann1996positively}, Mann shows that for the free prosolvable group on $k$ generators $\widehat\cF_k^{\textrm{sol}}$, $ m_N(\widehat\cF_k^{\textrm{sol}}) \le N^{\frac{13}{4} k+2}$.  This implies that the family of prosolvable groups (and, in particular the finite solvable groups) are uniformly PMSG. More generally, in \cite{borovik1996maximal}, Borovik--Pyber--Shalev showed the following: Let $L$ be a finite group.  Then there exists a constant $E(L)$ such that if $G$ is any profinite group of rank $k$  without a section $L$, then $m_N(G) \le N^{E(L)\textrm{rank}(G)} $ for every $n \in \mathbb{N}$.  Recall that $L$ is a section of $G$ if $G$ has an open subgroup which  is mapped onto $L$.  This  implies that the family 
  \[\frak{G}_{\neq L} = \{ G \hbox{\ profinite\ without\ } L-\hbox{section}\}
  \]
  is uniformly PMSG.
  Mann and Shalev showed in \cite{mann1996simple} that a profinite group $G$ is PMSG if and only if it is {\emph{ positively finitely generated}} (PFG), i.e. there exists a positive integer $r$, such that
\begin{equation*} \Pro_{x_1,...,x_r\sim G}[\overline{\langle x_1,...,x_r\rangle}=G] > 0,\end{equation*}
where $x\sim G$ is sampling an element according to the Haar measure of $G$.

\begin{thm} \label{thm:fin_sol_gps} Let $0 < \delta< \frac{1}{3},$ and $\frak{G}$ a family of profinite groups which are uniformly PMSG with constant $E'>0$.  Then, there exists a constant $C=C(E',\delta)$ such that every $G \in \frak{G}$ has a subset $A$ of size $|A|\leq C\cdot \textrm{rank}(G)$ with $\delta(A;G)\geq \delta$.  
\end{thm}

\begin{proof}  Let $G \in \frak{G}$ with $\textrm{rank}(G)=k$. The idea of the proof is as follows: If $x\sim G$ and $M$ is a subgroup of $G$ of index $N$, then $\Pro[x\in M]=\nicefrac{1}{N}$. Hence, the probability $M$ avoids less than a $\delta$-fraction of $n$ independently sampled elements of $G$ decreases exponentially. Since we have a polynomial bound on the number of maximal subgroups of each index, we can use a union bound and choose $n$ so that there is a positive probability for the sampled $n$ elements to have a detection probability of $\delta$.

Let $A$ be a set of $n$ independently sampled elements of $G$. 
Then, for a maximal subgroup $M$ of $G$ of index $N$, we have 
\[
\begin{split}
    \Pro[|A\cap M|\geq (1-\delta)n]\leq \binom{n}{(1-\delta)n}\cdot \frac{1}{N^{(1-\delta)n}}\leq \frac{2^{H(1-\delta)n}}{N^{(1-\delta)n}}\leq N^{(H_2(1-\delta)-(1-\delta))n},
\end{split}
\]
where $H_2(1-\delta)=-\delta\log\delta-(1-\delta)\log(1-\delta)$ is the binary entropy function. Therefore, by a union bound, 
\[
\begin{split}
  \Pro[\exists M\leq_{max} G \colon   |A\cap M|\geq (1-\delta)n]&\leq \sum_{N=2}^\infty\underbrace{ m_N(G)}_{\leq N^{E'k}}N^{(H_2(1-\delta)-(1-\delta))n}\\
  &\leq \sum_{N=2}^\infty N^{E'k+{(H_2(1-\delta)-(1-\delta))n}}.
\end{split}
\]
In particular, if we choose $n$ such that $E'k+{(H_2(1-\delta)-(1-\delta))n}\leq -2$, then this probability is smaller than $1$ and we are done. 
Solving the inequality, and as $H_2(1-\delta)<1-\delta$ for $\delta<\nicefrac{1}{3}$,  we deduce that $$n\geq \frac{-2-E'k}{H_2(1-\delta)-(1-\delta)}\geq \frac{-2E'}{H_2(1-\delta)-(1-\delta)}\cdot k.$$
Therefore,  we can choose $C=\frac{-2E'}{H_2(1-\delta)-(1-\delta)}$, which finishes the proof.
\end{proof}

Since finite solvable groups are uniformly PMSG, Theorem \ref{thm:intro_solvable_groups} is deduced from the above. For finite solvable groups, by Mann's result,  $E'\leq \frac{17}{4}$, and if we take $\delta=\nicefrac{1}{10}$, then $H(1-\delta)-(1-\delta)\leq -\nicefrac{2}{5}$ and $C\leq 17\cdot 5=85$.

\section{Running time and lower bounds on the length}\label{sec:running_time}
The running time of the tester $\cT$ which we described in the introduction is composed of its sampling phase and verification phase. In the sampling phase, $\cT$ uses as much time steps as random bits, with some constant overhead, to choose the index $i$. Then, the tester needs to calculate $w_i\in \cF_k$ out of the sampled index $i$, write down $w_i$, and send it to the oracle of the function $h$. Lastly, after $h$ replies with a bit, it takes a constant time to decide whether to accept or reject. Specifically, the time resources needed for running $\cT$ is at least the length of (the encoding of) $w_i$, which is $|w_i|\cdot \log k$.\footnote{We use $|w|$ to denote the length of $w\in \cF_k$ as a word in the fixed basis $B=\{x_1,...,x_k\}$ of $\cF_k$.}

We now prove that to achieve detection probability of at least $\delta>0$, we need $\Ex_{w_i\in A}[|w_i|]\geq \delta k$, which in turn means that just to write down an average $w_i$ we need $\omega (\textrm{poly}(\log k))$ time steps.  This means that such a tester cannot be time efficient. 

\begin{prop}\label{prop:running_time}
    Let $A\subseteq \cF_k$ be a finite set, and $\delta>0$, such that
    \[
\forall H\lneq \cF_k \ \colon \ \ \frac{|A\cap H|}{|A|}\leq 1-\delta.
    \]
    Then, 
    \[
    \Ex_{w\in A}[|w|]\geq \delta k.
    \]
\end{prop}  
\begin{proof}
    Let $B=\{x_1,...,x_k\}$ be a fixed basis of $\cF_k$. Let $H_i$ be the subgroup of $\cF_k$ generated by all of the $x_j$'s except for $x_i$. Then, $w\in H_i$ if and only if when writing it down as a reduced word in $B$, $x_i$ and its inverse do not appear. Therefore, 
    the condition $\frac{|A\cap H_i|}{|A|}\leq 1-\delta$ implies that $x_i$ or its inverse \textbf{appear} in at least $\delta|A|$ of the words in $A$. Since $|w|\geq \sum_{i=1}^k {\bf 1}_{x_i\in w}$, where $x_i\in w$ means that $x_i$ or its inverse appear in the word $w$, we have 
    \[
\Ex_{w\in A}[|w|]=\frac{1}{|A|}\sum_{w\in A}|w|\geq \frac{1}{|A|}\sum_{w\in A}\sum_{i=1}^k{\bf 1}_{x_i\in w}= \frac{1}{|A|}\sum_{i=1}^k\underbrace{\sum_{w\in A}{\bf 1}_{x_i\in w}}_{\geq\delta|A|}\geq \delta k.
    \]
\end{proof}

\section{Open Problems}\label{sec:open_problems}
This paper suggests developing a theory for \emph{non-commutative  error correcting codes} --- for  universal codes or any other $\frak{G}$-codes, where $\frak{G}$ is a natural subclass of groups. Most of the classical results in error correcting codes, can be formulated as problems in this setting. For example, 
\begin{problem}
     As random codes of linear size are good in the classical setup; do
random linearly sized subsets of the Hadamard universal codes  --- which was described in Theorem \ref{thm:exp_set} --- resolve Theorem \ref{conj:main}? Moreover, is there any linearly sized subset of the universal Hadamard code that induces a good universal code?
\end{problem}
\begin{problem}
     Various tradeoff results are known between the rate and distance of classical codes. What are the appropriate tradeoffs in the non-commutative case?  Specifically, are there new obstructions that arise in the universal codes case and not in the classical setup?
\end{problem}
\begin{problem}
    Are there good universal codes with words of linear length (at this point, Theorem \ref{thm:Spielman_construction} provides codes with words of polynomial length)?
\end{problem}
In response to our results, the following problem was raised independently by Nir Lazarovich and Zlil Sela:
\begin{problem}
    Does the following sampling scheme provides a good universal code with positive probability?
\begin{itemize}
    \item Fix a basis $B$ of $\cF_k$.
    \item Sample $C$ length $r$ walks  in the Cayley graph of the automorphism group $Aut(\cF_k)$ with respect to the product replacement algorithm generators (see \cite{lubotzky2001product}). Let their endpoints be $\varphi_1,...,\varphi_C\in Aut(\cF_k)$, and let $\varphi_0=\Id$. 
    \item Let $A=\bigcup_{i=0}^C\varphi_i(B)$. Then $A$ is of size $(C+1)|B|=(C+1)k$. Do we expect it to have a constant detection probability for some large enough $r=r(k)$? Say $r={\rm polylog}(k)$?
\end{itemize}
\end{problem}

\begin{problem}\label{prob:lifting_codes}
    Lift other  classical families of codes   to  universal codes.
\end{problem}
Specifically, with respect to Problem \ref{prob:lifting_codes},
in an early stage of this project we tried to provide a non-commutative version of the Reed--Muller (low individual degree polynomials) codes. We noticed that one can use \emph{mutual half basees} (which will be defined in a moment) to mimic polynomial interpolation on axis parallel lines. This led to the following problem which seems to be of independent interest.
\begin{defn}
    Let $k$ be a positive integer. A collection of \emph{mutual half bases for $\cF_{2k}$} is a set $\{B_0,...,B_q\}$ such that:
\begin{itemize}
    \item $|B_i|=k$ for every $0\leq i\leq q$.
    \item $B_i\cup B_j$ is a basis of $\cF_{2k}$ for every $0\leq i\neq j\leq q$.
\end{itemize}
\end{defn}

\begin{problem}
    What is the maximal number of mutual half bases in $\cF_{2k}$? Namely, how big can $q$ be in the above definition?

    In a similar fashion to that of Section \ref{sec:proper_subgps_and_ECCs}, this property is preserved by quotients, and by taking $\Phi_2$ the abelianization $\pmod 2$ we can bound $q$ by $2^k$. 
\end{problem}

\appendix
\section{Existence of lossless expanders}\label{sec:existence_lossless}

\begin{proof}[Proof of Lemma \ref{lem:unique_neighbors1}]
    Let $\Gamma=(L,R,E)$ be a uniformly random left $d$-regular graph with $|L|=n$ and $|R|=\beta n$. Let $S\subseteq L$ be a subset on the left where $|S|\leq \alpha n$ for some fixed $\alpha>0$, and $T\subseteq R$ be a subset on the right such that $|T|< (1-\eps)d|S|$. Let $X_{S,T}$
 be the indicator of whether the neigborhood of $S$ is contained in $T$, namely $N(S)=\{v\in R\mid \exists w\in S\colon vw\in E\}\subseteq T$. The probability $X_{S,T}=1$ is $\left(\frac{|T|}{|R|}\right)^{d|S|}$, which is bounded by $\left(\frac{(1-\eps)d|S|}{\beta n}\right)^{d|S|}$. So, by a union bound,
 \begin{equation}\label{eq:existence_of_expanders}
     \begin{split}
        \Pro\left[\sum X_{S,T}\geq 1\right]&\leq \sum_{S,T}\Pro[X_{S,T}=1]\\
        &\leq  \sum_{s=1}^{\alpha n}\sum_{|S|=s}\sum_{|T|=(1-\eps)ds} \left(\frac{(1-\eps)ds}{\beta n}\right)^{ds}\\
        &= \sum_{s=1}^{\alpha n}\binom{n}{s}\binom{\beta n}{(1-\eps)ds}\left(\frac{(1-\eps)ds}{\beta n}\right)^{ds}\\
        &\leq \sum_{s=1}^{\alpha n}\left(\frac{ne}{s}\right)^s \left(\frac{\beta ne}{(1-\eps)ds}\right)^{(1-\eps)ds}\left(\frac{(1-\eps)ds}{\beta n}\right)^{ds}\\
        &= \sum_{s=1}^{\alpha n} \left(e^{1+(1-\eps)d}\cdot n^{1-\eps d}\cdot s^{-1+\eps d} \cdot \beta^{-\eps d}\cdot ((1-\eps)d)^d\right)^s.
   \end{split}
 \end{equation}
    Now, as $\nicefrac{s}{n}\leq \alpha$ for every $s$ in the sum,  we can choose $d$ to be larger than  $e,2\eps^{-1}$ and $\beta^{-1}$, and every summand in \eqref{eq:existence_of_expanders} is bounded from above by $\left(\alpha^{\nicefrac{\eps}{2}}d^{3}\right)^{sd}$.
    Hence, by choosing $\alpha$ which is smaller than ${d}^{-\nicefrac{8}{\eps}}$, each summand is bounded by ${d^{-sd}}$ which translates to $\Pro\left[\sum X_{S,T}\geq 1\right]$ being bounded by $\frac{1}{d^d-1}<1$ and the existence of graphs for which all $X_{S,T}$ are zero. Hence, under the following choice of parameters
    \begin{equation}
d\geq e,2\eps^{-1},\beta^{-1}\quad,\quad \alpha\leq   d^{\nicefrac{-8}{\eps}}\quad,\quad n_0\geq \alpha^{-1}\ ,
    \end{equation}
    there are $(d,\beta,\alpha,1-2\eps)$ unique neighbor expanders for every $n\geq n_0$ such that $\beta n$ is an integer.
\end{proof}

\bibliographystyle{plain}
\bibliography{Bibliography}

\end{document}